\documentclass[a4paper,12pt,reqno,twoside]{amsart}

\usepackage{amssymb,amsmath}
\usepackage[foot]{amsaddr}
\usepackage[colorlinks,citecolor=blue,urlcolor=blue]{hyperref}
\usepackage[hmargin=2.5cm,vmargin=2.5cm]{geometry}
\usepackage{enumitem}

\usepackage{tikz}
\usetikzlibrary{automata,arrows,positioning,calc}

\newcommand{\floor}[1]{\lfloor#1\rfloor}

\DeclareMathOperator{\diam}{diam}
\DeclareMathOperator{\bP}{\mathbb{P}}
\DeclareMathOperator{\bE}{\mathbb{E}}

\newcommand{\eps}{\varepsilon}

\newcommand{\LL}{{\mathrm{LL}}}
\newcommand{\Lip}{{\mathrm{Lip}}}

\newcommand\bN{{\mathbb N}}
\newcommand\bR{{\mathbb R}}
\newcommand\bZ{{\mathbb Z}}

\newcommand{\hr}{{\hat{r}}}
\newcommand{\hh}{{\hat{h}}}
\newcommand{\htau}{{\hat{\tau}}}

\newcommand\cB{{\mathcal B}}
\newcommand\cC{{\mathcal C}}
\newcommand\cF{{\mathcal F}}
\newcommand\cT{{\mathcal T}}
\newcommand\cP{{\mathcal P}}

\newcommand{\tH}{{\widetilde H}}

\newcommand{\tX}{{\widetilde X}}
\newcommand{\tmu}{{\tilde{\mu}}}
\newcommand{\tomega}{{\tilde{\omega}}}

\numberwithin{equation}{section}

\newtheorem{thm}{Theorem}[section]
\newtheorem{prop}[thm]{Proposition}
\newtheorem{cor}[thm]{Corollary}
\newtheorem{lem}[thm]{Lemma}
\newtheorem{defn}[thm]{Definition}

\theoremstyle{remark}
\newtheorem{rmk}[thm]{Remark}

\begin{document}

\title[Loss of memory and moment bounds]{Loss of memory and
moment bounds for nonstationary intermittent dynamical systems}

\author{A.~Korepanov$^1$}
\address{$^1$College of Engineering, Mathematics and Physical Sciences,
University of Exeter, Exeter, EX4 4QF, UK}
\email{a.korepanov@exeter.ac.uk}

\author{J.~Lepp\"anen$^2$}
\address{$^2$Laboratoire de Probabilit\'es, Statistique et Mod\'elisation (LPSM),
CNRS, Sorbonne Universit\'e, Universit\'e de Paris, 4 Place Jussieu, 75005 Paris, France}
\email{leppanen@lpsm.paris}

\maketitle

\begin{abstract}
    We study nonstationary intermittent dynamical systems, such as compositions
    of a (deterministic) sequence of Pomeau-Manneville maps.
    We prove two main results: sharp bounds on memory loss, including the ``unexpected''
    faster rate for a large class of measures, and sharp moment bounds for
    Birkhoff sums and, more generally, ``separately H\"older'' observables.
\end{abstract}

\section{Introduction}

Suppose that $X$ is a measurable space and $T_n \colon X \to X$, $n \geq 1$,
is a sequence of transformations; let $T_{1,n} = T_n \circ \cdots \circ T_1$.
Let $v_n \colon X \to \bR$, $n \geq 0$, be a
sequence of observables. Consider a process such as the Birkhoff sum
\[
    S_n = v_0 + v_1 \circ T_{1,1} + \cdots + v_{n-1} \circ T_{1,n-1}
\]
or the record process
\[
    M_n = \max \{v_0, v_1 \circ T_{1,1}, \ldots, v_{n-1} \circ T_{1,n-1}\}
    .
\]
Such processes are the central objects of interest in nonstationary dynamical systems.
Often the initial state is random (we are given a probability measure on $X$),
then we think of $S_n$ and $M_n$ as random processes.

Statistical properties of the above processes have been a topic of very intense recent investigations.
Under various assumptions, numerous authors published results on:
\begin{itemize}
    \item rates of memory loss (or decay of correlations)~\cite{BB16,BBD14,SVZ20,SYZ13},
    \item ergodic theorems, central limit theorems, local limit theorems
        and stable laws~\cite{ANV15,BB16,CR07,DFGTV18.2,DFGTV20,LS16,NPT19,NTV18},
    \item almost sure invariance principles~\cite{ANV15,DFGTV18.1,H19,HNTV17,SVZ20,S19},
    \item large deviations and concentration inequalities~\cite{ANV15,AR16,DFGTV18.2,DFGTV20,NPT19,S19},
    \item exponential law for hitting times~\cite{HRY17} and extreme value laws~\cite{FFV18}.
\end{itemize}
This list is not exhaustive.

In this paper we are interested in nonstationary dynamical systems with intermittency, as in
the Pomeau-Manneville~\cite{PM80} scenario. These are chaotic (turbulent) systems with a regular (laminar)
region, in which a trajectory can be trapped for a very long time.
Under natural assumptions we prove optimal asymptotic bounds for:
\begin{itemize}
    \item Memory loss: $\bigl|(T_{1,n})_* \mu - (T_{1,n})_* \nu\bigr|$, where $\mu$ and $\nu$ are
        probability measures, $|\cdot|$ denotes the total variation
        and $(\cdot)_*$ is the pushforward.
    \item Moment bounds: $\bE|S_n - \bE S_n|^p$ for $p > 0$.
\end{itemize}
In the abstract setting our results are presented much later, in Theorems~\ref{thm:decdec} and~\ref{thm:mom}.
Since the abstract setting is not suitable for an introduction, here we present specific applications
to the most standard and popular example: the Liverani--Saussol--Vaienti~\cite{LSV99} maps
$T \colon [0,1] \to [0,1]$,
\begin{equation}
    \label{eq:LSV}
    T(x) = \begin{cases}
        x(1 + 2^\gamma x^\gamma), & x \leq 1/2,
        \\
        2x - 1,                   & x > 1/2.
    \end{cases}
\end{equation}
Here $\gamma \in (0,1)$ is a parameter. These maps often serve
as a prototypical example of slowly (polynomially) mixing systems.
We recommend Gou\"ezel~\cite{G04c} for some background information
on their statistical behavior.

Theorems~\ref{thm:main:memory} (memory loss) and~\ref{thm:main:moments} (moment bounds)
are applications of Theorems~\ref{thm:decdec} and~\ref{thm:mom} respectively.
They illustrate the strength of our method and,
we hope, give our reader an intuitive understanding of this paper.

Let $T_1, T_2, \ldots$ be a sequence of maps~\eqref{eq:LSV} corresponding to
parameters $\gamma_1, \gamma_2, \ldots$, and suppose that $\sup_n \gamma_n \leq \gamma^*$
with a fixed $\gamma^* \in (0,1)$.
As in~\cite{LSV99}, for some $a > 2^{\gamma^*} (\gamma^* + 2)$ we let
\begin{equation}
    \label{eq:cone}
    \begin{aligned}
        \cC_* = \Bigl\{ f\in C((0,1])\cap L^1 :
            \text{} & \text{$f \ge 0$, $f$ is decreasing,} 
            \\
            & \text{$x^{\gamma^* + 1}f(x)$ is increasing,
            $f(x)\le a x^{-\gamma^*} \int_0^1 f(y) \; dy $} \Bigr\} 
            .
    \end{aligned}
\end{equation}
Then $\cC_*$ is a convex cone of functions,
containing densities of all absolutely continuous probability measures
invariant under maps~\eqref{eq:LSV} with parameters in $(0, \gamma^*]$.

\begin{thm}
    \label{thm:main:memory}
    Suppose that $\mu$ and $\mu'$ are probability measures on $[0,1]$ with H\"older densities
    and $\nu$ is a probability measure on $[0,1]$ with density in $\cC_*$.
    Then:
    \begin{enumerate}[label=(\alph*)]
        \item\label{thm:main:memory:normal}
            $|(T_{1,n})_* (\mu - \nu)| = O(n^{- 1/\gamma^* + 1})$,
        \item\label{thm:main:memory:unexpected}
            $|(T_{1,n})_* (\mu - \mu')| = O(n^{- 1/\gamma^*})$.
    \end{enumerate}
\end{thm}

Let $v_n \colon [0,1] \to \bR$ be a family of H\"older continuous observables
with uniformly bounded H\"older norm, i.e.\ 
$\sup_n \|v_n\|_\eta < \infty$ for some $\eta \in (0,1]$,
where
\( \|v\|_\eta = \sup_x |v(x)| + \sup_{x \neq y} |v(x) - v(y)| / |x-y|^\eta \).
Let $\mu$ be a probability measure with density in $\cC_*$.
On the probability space $([0,1], \mu)$, define a random process
\[
    V_n
    = v_0 + v_1 \circ T_{1,1} + \cdots + v_{n-1} \circ T_{1, n-1}
    .
\]
Let $S_n = V_n - \bE V_n$ and $S_n^* = \max_{k \leq n} |S_k|$.

\begin{thm}
    \label{thm:main:moments}
    Let $n \geq 0$.
    \begin{enumerate}[label=(\alph*)]
        \item If $\gamma^* \in (0,1/2)$, then
            \[
                \bE (S_n^*)^{2(1/\gamma^* - 1)}
                \leq C n^{1/\gamma^* - 1}
                .
            \]
        \item If $\gamma^* = 1/2$, then
            \[
                \bE (S_n^*)^2
                \leq C n \log (n + 1)
            \]
            and for $p > 2$,
            \[
                \bE (S_n^*)^p
                \leq C_p n^{p-1}
                .
            \]
        \item If $\gamma^* \in (1/2,1)$, then for all $t > 0$,
            \[
                \bP(S_n^* \geq t) \leq C n t^{-1/\gamma^*}
                .
            \]
    \end{enumerate}
    Here $C$ denotes constants which depend only on $\gamma^*$ and $\sup_n \|v\|_\eta$,
    and $C_p$ depends in addition on $p$.
\end{thm}

\begin{rmk}
    The moment bounds from Theorem~\ref{thm:main:moments}, together with $\|S_n^*\|_\infty \leq C n$,
    can be used to obtain optimal bounds on $\bE(S^*_n)^p$ for all $p \in [1, \infty)$, as it is done in
    Gouez\"el and Melbourne~\cite{GM14}.
\end{rmk}

\begin{rmk}
    Theorem~\ref{thm:main:moments} is stated for Birkhoff sums. We note that its abstract counterpart,
    Theorem~\ref{thm:mom}, is stated for \emph{separately H\"older} observables,
    of which Birkhoff sums are a particular case.
\end{rmk}

The paper is organized as follows. In Section~\ref{sec:discussion} we comment on our
results. In Section~\ref{sec:results} we state the abstract versions of
Theorems~\ref{thm:main:memory} and~\ref{thm:main:moments}.
Sections~\ref{sec:mixing},~\ref{sec:mom:new} and~\ref{sec:mart} contain the proofs.

\section{Discussion}
\label{sec:discussion}

\subsection{Theorem~\ref{thm:main:memory}}

The two bounds in Theorem~\ref{thm:main:memory} are known in the contexts of homogeneous Markov chains,
see Lindvall~\cite{Li79} and references therein, and of stationary dynamical systems,
see Gou\"ezel~\cite{G04}.
In the nonstationary case, prior methods do not apply and our result is new.
We improve the best previously known
bound $O\bigl(n^{-1/\gamma^* + 1} (\log n)^{1/\gamma^*}\bigr)$ by Aimino, Hu, Nicol, T\"or\"ok
and Vaienti~\cite{AHNTV15}.

For a stationary dynamical system, Theorem~\ref{thm:main:memory}\ref{thm:main:memory:unexpected}
is new in the sense that the implied constant is explicit in its dependence on basic parameters
of a dynamical system, see Theorem~\ref{thm:decdec}.

A case of special interest is when the parameters $\gamma_n$ are random,
say independently and uniformly distributed in an
interval $[\gamma^-, \gamma^+]$. Then one expects the memory loss to correspond to the
quickest mixing map (i.e.\ the one for $\gamma^-$) for almost every sequence of parameters.
For the maps~\eqref{eq:LSV} such results are proved by Bahsoun, Bose and Ruziboev~\cite{BBR19}
with rate $O(n^{- 1 / \gamma^- + 1 + \delta})$ for every $\delta > 0$.
In contrast, we work in the worst case scenario, i.e.\ our bounds hold
for \emph{every} sequence of parameters and correspond to the slowest mixing map.
We conjecture that the bound of~\cite{BBR19} can be improved to at least $O(n^{- 1 / \gamma^- + \delta})$
for measures with H\"older densities, as in~Theorem~\ref{thm:main:memory}\ref{thm:main:memory:unexpected}.

\subsection{Theorem~\ref{thm:main:moments}}
In the stationary case, versions of Theorem~\ref{thm:main:moments} can be found in
Gou\"ezel and Melbourne~\cite{GM14} and in Dedecker and Merlev\`ede~\cite{DM15}.
These moment bounds are known to be optimal (see~\cite{GM14}),
hence our results are optimal as well.

\begin{rmk}
    While Theorem~\ref{thm:main:moments} gives optimal bounds for a general
    measure $\mu$ with density in $\cC_*$, it is natural to ask if one can get better
    bounds for nice measures, such as Lebesgue. We do not answer this question directly,
    yet we refer the reader to Dedecker, Gou\"ezel and Merlev\`ede~\cite[Section 3]{DGM18},
    where lower bounds on tails of Birkhoff sums are obtained for examples of related models:
    Markov chains and Young towers with polynomial tails.
    Their proof is written for probability measures starting on the ``base'' of
    the tower, which roughly corresponds to the Lebesgue measure for the maps~\eqref{eq:LSV},
    and their lower bounds are $\bP(S_n \geq x) \geq C n / x^p$ for all
    $c_1 n^{1/p} < x < c_2 n$, where $p$ corresponds to our $1/\gamma$.
    This hints that our bounds cannot be improved for measures such as Lebesgue.
\end{rmk}

As in~\cite{GM14}, we prove concentration bounds not only for Birkhoff sums, but for
a more general class of separately Lipschitz (or separately H\"older) functions on $[0,1]^\bN$,
see Theorem~\ref{thm:mom} and Remark~\ref{rmk:Lipschitz}.

Theorem~\ref{thm:main:moments} improves the moment bounds in
Nicol, Pereira and T\"or\"ok~\cite{NPT19} and Su~\cite{S19},
and implies the following bounds on large and moderate deviations:

\begin{cor}
    \label{cor:dev}
    In the notation of Theorem~\ref{thm:main:moments}, for every $p > 2$,
    \begin{equation}
        \label{eq:LD}
        \begin{aligned}
            \mu \bigl\{ | S_n / n | \geq \eps \bigr\}
            \le
            \begin{cases}
                C \eps^{- 2 ( 1/\gamma^* - 1) }  n^{ - (1/\gamma^* - 1) }, & \gamma^* \in (0,1/2)
                , \\
                C_p \eps^{-p} n^{-1}, & \gamma^* = 1/2
                , \\
                C \eps^{-1/\gamma^*} n^{-(1/\gamma^* - 1)}, & \gamma^* \in (1/2, 1)
                .
            \end{cases}
        \end{aligned}
    \end{equation}
    Further, for $\tau > 0$,
    \begin{equation}
        \label{eq:MD}
        \begin{aligned}
            \mu \bigl\{ | S_n / n^\tau | \geq \eps \bigr\}
            \le 
            \begin{cases}
                C \eps^{- 2 ( 1/\gamma^* - 1) }  n^{ - (2 \tau - 1) (1/\gamma^* - 1) }, & \gamma^* \in (0,1/2)
                , \\
                C \eps^{-2} n^{-(2\tau - 1)} \log (n+1), & \gamma^* = 1/2
                , \\
                C \eps^{-1/\gamma^*} n^{-(\tau/\gamma^* - 1)}, & \gamma^* \in (1/2, 1)
                .
            \end{cases}
        \end{aligned}
    \end{equation}
\end{cor}

Compared to results for stationary dynamics,~\eqref{eq:LD} agrees with the
optimal large deviation bounds, see Melbourne~\cite{M09} and also Pollicott and Sharp~\cite{PS09}.
In turn,~\eqref{eq:MD} is as good as one can infer from moment bounds,
but otherwise for $\gamma^* \in (0,1/2]$ there are more interesting inequalities, see
Dedecker, Gou\"ezel and Merlev\`ede~\cite{DGM18}.

In the nonstationary case,~\eqref{eq:LD} is a slight improvement over the bound
\[
    \mu \bigl\{ | S_n / n | \geq \eps \bigr\}
    \le C_p n^{-(1/\gamma^* - 1)} (\log n)^{1/\gamma^*} \eps^{-2p}
    \qquad \text{ for each } p > \max\{1, 1/\gamma^* - 1\}
\]
from~\cite[Theorem~4.1]{NPT19}. We remove the logarithmic term, get a better power
of $\eps$ when $\gamma^* \in (1/2, 1)$ and allow the observables $v_n$ to depend on $n$.

\subsection{Quasistatic dynamical systems}

The original motivation for our project is a question from quasistatic dynamical systems (QDS).
These are a class of nonstationary dynamical systems introduced by Dobbs and Stenlund~\cite{DS16}
to model situations where external influences cause the observed system to transform slowly over time.
We refer the reader to~\cite{DS16} for the abstract definition
of the model and discussion on its physical significance. A special class of QDSs
described by the intermittent family~\eqref{eq:LSV} was studied 
by Lepp\"anen and Stenlund~\cite{L18,LS16}:
the evolution of states is described by compositions of the form 
\begin{align*}
    x_{n, k}
    = T_{\gamma_{n,k}} \circ \cdots \circ T_{\gamma_{n,1}}(x),    \qquad 0 \le k \le n,
\end{align*}
where $T_{\gamma_{n,k}}$ is the map~\eqref{eq:LSV} with parameter $\gamma_{n,k} \in (0,1)$,
and $\{\gamma_{n,k} : 0 \leq k \leq n\}$ is a triangular array such that
\begin{align}\label{eq:qds-approx}
    \lim_{n\to\infty} \gamma_{n, \lfloor nt \rfloor }
    = \Gamma_t
    ,
\end{align}
where $\Gamma \colon [0, 1] \to  (0, 1)$ is a sufficiently regular curve.
Starting from an initial state $x \in X = [0,1]$, $x_{n,k}$ is the state of the system after
$k$ steps on the $n$-th level of the array $\{ \gamma_{n,k}\}$.
The levels of the array approximate $\Gamma$ ever more accurately as $n$ increases.
Hence the intermittent QDS is a setup of intermittent systems with slowly changing parameters.
Given an initial distribution $\mu$ for $x$, one is interested in the statistical properties of 
$(x_{n,k})_{k=0}^n$ in the limit $n\to\infty$.

Let $v : X \to \bR$ be a Lipschitz continuous observable. Consider the fluctuations $\xi_{n} \colon X \times [0,1] \to \bR$ 
defined by
\begin{align*}
\xi_n(x,t)
& = n^{-\frac12}  \biggl[ S_n(x,t) - \int_0^1 S_n(x,t) \, d \mu(x) \biggr]; \\
S_n(x,t)
        & =  \int_0^{nt} v(x_{n, \floor{s}}) \, ds.
\end{align*}
One may view $\xi_n(x,t)$ as a random element in the space $C[0,1]$ of continuous functions.
Under the assumptions that (a) $\Gamma_t$ is H\"{o}lder continuous with
$ \Gamma_t \leq \gamma^* < 1/3$, (b) the density of $\mu$ belongs to the cone $\cC_*$, and 
(c) the convergence~\eqref{eq:qds-approx} happens polynomially
fast and uniformly in $t$, it was shown in~\cite{L18} that 
$\xi_n$ converges in distribution to $\chi(t) = \int_0^t \sigma_s(v) \, dW_s$.
Here $s \mapsto \sigma_s(v)$ is a deterministic nonnegative continuous function and  $W$
is a standard Brownian motion. Theorem~\ref{thm:main:moments} allows us to extend this
result from $\gamma^* < 1/3$ to $\gamma^* < 1/2$. Indeed, by \cite[Theorem~1.3]{L17},
it suffices to show that $\xi_n$ are tight in $C[0,1]$, which follows by the Kolmogorov
criterion since Theorem~\ref{thm:main:moments} implies the existence of a small
$\varepsilon > 0$ such that 
\begin{align*}
    \int_0^1 \bigl|  \xi_n( x,  t + \delta ) - \xi_n(x, t)  \bigr|^{2 + \varepsilon}  \, d\mu(x) 
    = O(\delta^{1 + \frac{\varepsilon}{2}})
\end{align*}
as $n \to \infty$, whenever $0 \le t \le t + \delta \le 1$.

Alternatively, one can use the moment bounds from~\cite{NPT19} or~\cite{S19},
but these were not available when we started this project.

\subsection{Mixing}

On early stages of this project we attempted to prove Theorem~\ref{thm:main:moments}
without relying on mixing properties of the maps. For stationary dynamics,
there exist proofs which give close to optimal moment bounds~\cite{KKM18} which do not depend
on the speed of mixing, and moreover do not need mixing at all.
We found, however, that mixing is indispensable in the nonstationary
setup. Problems appear already when a dynamical system is fixed but observables are
changing. As a simple example of such system, consider the Markov chain
$g_0, g_1, \ldots$ on state space $\{A, B, C\}$ with $g_0$ distributed according to some
probability measure and the following transition probabilities:
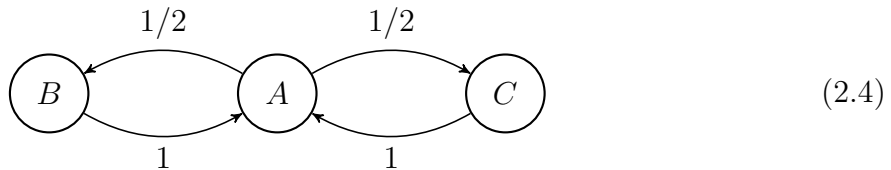
\begin{equation}
    \label{eq:MC}
    \begin{aligned}
        \begin{tikzpicture}[->, >=stealth', auto, semithick, node distance=3cm]
            \tikzstyle{every state}=[fill=white,draw=black,thick,text=black,scale=1]
            \node[state]    (A)                     {$A$};
            \node[state]    (B)[left of=A]    {$B$};
            \node[state]    (C)[right of=A]   {$C$};
            \path (A) edge[bend right]    node[above]{$1/2$}       (B);
            \path (A) edge[bend left]     node[above]{$1/2$}       (C);
            \path (B) edge[bend right]    node[below]{$1$}         (A);
            \path (C) edge[bend left]     node[below]{$1$}         (A);
        \end{tikzpicture}
    \end{aligned}
\end{equation}
This Markov chain is 2-periodic and thus not mixing.
Let $v_n \colon \{A, B, C\} \to \bR$ and $S_n = \sum_{j=0}^{n-1} v_j(g_j)$.
If $v_n$ do not depend on $n$, then $n^{-1/2} (S_n - \bE S_n )$ converges weakly
to a normal random variable. But
\[
    \text{if} \quad
    v_n(g) = 
    \begin{cases}
        (-1)^{n+1}, & g = A
        \\
        (-1)^{n}, & g \in \{B,C\}
    \end{cases}
    , \quad \text{then} \quad
    S_n = 
    \begin{cases}
        -n, & g_0 = A
        \\
        n,  & g_0 \in \{B,C\}
    \end{cases}
    .
\]
Then $S_n$ does not satisfy any interesting concentration inequalities.

Also, we found that for $\gamma^* \in (0,1/2)$, Theorem~\ref{thm:main:moments} can be proved
using memory loss with asymptotics $O(n^{-1/\gamma^* + 1})$ as in
Theorem~\ref{thm:main:memory}\ref{thm:main:memory:normal},
and close to optimal results can be obtained with the slightly weaker bound
$O\bigl(n^{-1/\gamma^* + 1} (\log n)^{1/\gamma^*}\bigr)$ from~\cite{AHNTV15},
as it is done in~\cite{NPT19}.
For $\gamma^* \in (1/2, 1)$ the situation is significantly more complicated. We guess that 
the bound $O(n^{-1/\gamma^* + 1})$ would suffice for Birkhoff sums, see~\cite[Proposition~A.1]{DM15}.
But for the generality of separately H\"older observables we do not see a way around
Theorem~\ref{thm:main:memory}\ref{thm:main:memory:unexpected}, which is unfortunate
because it is significantly harder to prove than
Theorem~\ref{thm:main:memory}\ref{thm:main:memory:normal}.
Luckily, it is also more interesting.


\section{Abstract setup and results}
\label{sec:results}

\subsection{Nonstationary nonuniformly expanding dynamical system}
\label{sec:nnue}

Let $(X,d)$ be a metric space which is bounded, separable and universally
measurable.\footnote{Most spaces are universally measurable, see Shortt~\cite{S84}.}
We endow $X$ with the Borel sigma-algebra, and we only work with measurable sets.

Let $Y \subset X$ and let $m$ be a probability measure on $X$ with $m(Y) = 1$.
Let $\cT$ be a class of measurable transformations of $X$.
Given a sequence of transformations $T_1, T_2, \ldots$, we denote
$T_{k,\ell} = T_\ell \circ \cdots \circ T_k$. (If $k > \ell$,
then $T_{k, \ell}$ is the identity map.)

For a nonnegative measure $\mu$ on $Y$ with density $\rho = d\mu / dm$, we denote
by $|\mu|_{\LL}$ the Lipschitz seminorm of the logarithm of $\rho$:
\[
    |\mu|_{\LL}
    = \sup_{y \neq y' \in Y} \frac{|\log \rho(y) - \log \rho(y')|}{d(y,y')}
    ,
\]
with a convention that $\log 0 = -\infty$ and $\log 0 - \log 0 = 0$.

We suppose that there exist constants $\lambda > 1$, $K > 0$, $\delta_0 > 0$ and $n_0 \geq 1$,
and a function $h \colon \{0,1,\ldots\} \to [0,\infty)$ such that the following assumptions
hold for each sequence $T_1, T_2 \ldots \in \cT$.

For $x \in X$, let
\[
    \tau(x)
    = \inf \{n \geq 1 : T_{1,n} (x) \in Y \}
\]
be the first return time to $Y$.
First, we assume that there is a finite or countable partition $\cP$ of $X$,
up to an $m$-zero measure set, such that $Y$ is $\cP$-measurable and for each $a \in \cP$:
\begin{enumerate}[label=(NU:\arabic*)]
    \item $m(a) > 0$.
    \item $\tau$ is constant on $a$ with value $\tau(a)$.
    \item If $a \subset Y$, then the map $F_a = T_{1,\tau(a)} \colon a \to Y$ is a bijection,
        and for all $y,y' \in a$,
        \[
            d(F_a(y), F_a(y')) \geq \lambda d(y,y')
            .
        \]
        Further, $F_a$ is nonsingular with log-Lipschitz Jacobian:
        \[
            \zeta = \frac{d (F_a)_* (m|_a)}{dm}
            \quad \text{satisfies} \quad
            |\zeta|_{\LL} \leq K
            .
        \]
    \item For all $x,x' \in a$, with $F_a = T_{1, \tau(a)}$ as above,
        \[
            \max_{0 \leq j \leq \tau(a)} d(T_{1,j}(x), T_{1,j}(x'))
            \leq K d(F_a(x), F_a(x'))
            .
        \]
\end{enumerate}
In other words, the first return map $y \mapsto T_{1, \tau(y)}(y)$ is full branch Gibbs-Markov,
and returns from outside of $Y$ have bounded backward expansion.

Next, to quantify mixing we assume that:
\begin{enumerate}[label=(NU:\arabic*),resume]
    \item $m(\tau \geq n) \leq h(n)$ for all $n$.
    \item\label{ass:integrable}
        $\sum_{n = 1}^\infty h(n) \leq K$.
    \item\label{ass:mixing}
        $m(T_{1,n}^{-1}(Y)) \geq \delta_0$ for every $n \geq n_0$.
\end{enumerate}

\begin{rmk}
    Since $\int \tau \, dm = \sum_{n \geq 1} m(\tau \geq n) \leq \sum_{n \geq 1} h(n)$,
    assumption~\ref{ass:integrable} guarantees that the return times $\tau$, parametrized by
    sequences of maps, are uniformly integrable.
\end{rmk}

\begin{rmk}
    To satisfy assumption~\ref{ass:mixing}, it is sufficient that (a) other assumptions hold and
    (b) there exist $\delta_0' > 0$ and coprime integers $p_1, p_2, \ldots, p_N$ such that
    $m(\tau = p_n) \geq \delta_0'$ for each $n$.
    The proof repeats that for the stationary dynamics, see~\cite[Section~4.2]{KKM19}.
\end{rmk}

\begin{rmk}
    \label{rmk:Lipschitz}
    In papers on nonuniformly expanding maps one usually assumes that the
    Jacobian is log-H\"older. We assume log-Lipschitz purely to simplify notation:
    we do not lose generality. If we let $d_\eta(x,y) = d(x,y)^\eta$ with $\eta \in (0,1)$,
    then $d_\eta$ is also a metric, all our assumptions are satisfied on $(X, d_\eta)$ with
    slightly different constants, and $\eta$-H\"older functions with respect to $d$
    are Lipschitz with respect to $d_\eta$.
\end{rmk}

\subsection{Memory loss}

\begin{prop}
    \label{prop:K}
    There exist constants $0 < K_1 < K_2$, depending only on
    $\lambda$ and $K$, such that for each sequence
    $T_1, T_2 \ldots \in \cT$ with the corresponding partition $\cP$
    and return time $\tau$, for each nonnegative measure $\mu$ on $Y$
    with $|\mu|_{\LL} \leq K_2$ and each $a \in \cP$, $a \subset Y$,
    \[
        \bigl| (T_{1,\tau(a)})_* (\mu|_a) \bigr|_{\LL} \leq K_1
        .
    \]
    The constants $K_1$, $K_2$ can be chosen arbitrarily large.
\end{prop}

\begin{proof}
    It is standard, see e.g.\ \cite[Proposition~3.1]{KKM19}, that
    \[
        \bigl| (T_{1,\tau(a)})_* (\mu|_a) \bigr|_{\LL}
        \leq K + \lambda^{-1} |\mu|_{\LL}
        .
    \]
    We can choose any $K_2 > (1-\lambda^{-1})^{-1} K$ and $K_1 = K + \lambda^{-1} K_2$.
\end{proof}

Fix $K_1$, $K_2$ as in Proposition~\ref{prop:K}.

\begin{defn}
    We say that a nonnegative measure $\mu$ on $X$ is \emph{regular}
    if for every $T_1, T_2, \ldots \in \cT$ with the corresponding partition $\cP$
    and every $a \in \cP$,
    \[
        \bigl| (T_{1,\tau(a)})_* (\mu|_{a}) \bigr|_{\LL}
        \leq K_1
        .
    \]
    We say that $\mu$ has \emph{tail bound} $r$, with $r \colon \{0,1,\ldots\} \to [0,\infty)$,
    if for all $n \geq 0$,
    \[
        \mu \bigl( \{ x \in X : T_{1,k}(x) \not \in Y \text{ for all } 1 \leq k < n \}\bigr)
        \leq r(n)
        .
    \]
\end{defn}

\begin{rmk}
    The measure $m$ is regular with tail bound $r(n) = h(n)$, and every
    measure $\mu$ on $Y$ with $|\mu|_{\LL} \leq K_2$ is regular with tail bound
    $r(n) = \mu(Y) e^{K_2} h(n)$.
\end{rmk}

\begin{rmk}
    \label{rmk:forward-tails}
    Let $T_1, T_2, \ldots \in \cT$ and suppose that $h(n) = C n^{-\beta}$ with $\beta > 1$.
    If $\mu$ is a regular measure with tail bound  $r(n) = C n^{-\beta}$, then
    $(T_{1,k})_* \mu$ has tail bound $r_k(n) = C' k n^{-\beta}$, with $C'$ independent of $k$.
    If $r(n) = n^{-\beta + 1}$, then 
    $(T_{1,k})_* \mu$ has tail bound $r_k(n) = C' n^{-\beta + 1}$, again with $C'$ independent of $k$.
    See Proposition~\ref{prop:onedec} and Corollary~\ref{cor:const-tail}.
\end{rmk}

The abstract version of Theorem~\ref{thm:main:memory} is:

\begin{thm}
    \label{thm:decdec}
    Suppose that $\mu$ is a regular probability measure with tail bound $r$.
    Then for each sequence $T_1, T_2, \ldots \in \cT$, there exists a decomposition
    \[
        \mu = \sum_{n=1}^\infty \alpha_n \mu_n,
    \]
    where $\mu_n$ are probability measures and $\alpha_n$ are nonnegative constants
    with $\sum_{n \geq 1} \alpha_n = 1$ such that
    $(T_{1,n})_* \mu_n = m$ for each $n$. The sequence $\alpha_n$
    is fully determined by $K_1$, $K_2$, the constants in the
    definition of nonstationary nonuniformly expanding dynamical system
    ($\diam X$, $K$, $\lambda$, $n_0$, $\delta_0$), and
    the functions $h$ and $r$.
    In particular, $\alpha_n$ does not depend on $\mu$
    in any other way.

    \begin{itemize}
        \item 
            If $h(n) \leq C_\beta n^{-\beta}$ with $\beta > 1$
            and $r(n) \leq C'_\beta n^{-\beta'}$ with $\beta' \in (0, \beta]$,
            then 
            \[
                \sum_{j \geq n} \alpha_j
                \leq C C'_\beta n^{-\beta'}
                ,
            \]
            where $C$ depends only on $C_\beta$, $\beta$, $\beta'$, $K_1$, $K_2$
            and $\diam X$, $K$, $\lambda$, $n_0$, $\delta_0$ (i.e.\ on everything except
            $C'_\beta$).
        \item
            If $h(n) \leq C_\beta n^{-\beta}$ with $\beta > 1$,
            then for $n \geq 2 n_0$,
            \[
                \sum_{j \geq n} \alpha_j
                \leq r(\floor{n / 2} - n_0) + C n^{-\beta} \sum_{j=1}^\infty r(j)
                ,
            \]
            where $C$ depends only on $C_\beta$, $\beta$, $K_1$, $K_2$,
            and $\diam X$, $K$, $\lambda$, $n_0$, $\delta_0$.
        \item
            If $h(n), r(n) \leq C_\beta \exp( - C_\beta' n^\beta)$ with $\beta \in (0,1]$
            and $C_\beta, C'_\beta > 0$, then
            \[
                \sum_{j \geq n} \alpha_j
                \leq C \exp (- C' n^\beta)
                ,
            \]
            where $C$ and $C'$ depend only on $C_\beta$, $C_\beta'$, $\beta$, $K_1$, $K_2$,
            and $\diam X$, $K$, $\lambda$, $n_0$, $\delta_0$.
    \end{itemize}
\end{thm}

Theorem~\ref{thm:decdec} is proved in Section~\ref{sec:mixing}.

\begin{rmk}
    \label{rmk:decdec}
    If $\mu$ and $\mu'$ are regular probability measures as in Theorem~\ref{thm:decdec}, then
    \[
        \bigl|(T_{1,n})_* (\mu - \mu')\bigr|
        \leq 2 \sum_{j > n} \alpha_j
        .
    \]
\end{rmk}

\begin{cor}
    \label{cor:decdec}
    Let $\mu$ and $\mu'$ be regular probability measures as in Theorem~\ref{thm:decdec}.
    Let $\Theta \colon X \times X \to \{0,1,2,\ldots \} \cup \{\infty\}$,
    \[
        \Theta(x, x')
        = \inf \{k \geq 0 : T_{1,k}(x) = T_{1,k}(x') \}
        .
    \]
    Then there exists a probability measure $\tmu$ on $X \times X$ with marginals
    $\mu$ and $\mu'$ on the first and second coordinate respectively such that
    \[
        \tmu(\Theta \geq n)
        \leq \sum_{j \geq n} \alpha_j
        .
    \]
\end{cor}

\begin{proof}
    By Theorem~\ref{thm:decdec}, we have the decompositions
    $\mu = \sum_{n \geq 1} \alpha_n \mu_n$
    and $\mu' = \sum_{n \geq 1} \alpha_n \mu'_n$. Fix $n$.
    
    Let $F_n \colon X \to X \times X$, $F_n(x) = (x, T_{1,n}(x))$, and let
    $\nu = (F_n)_* \mu_n$. Then the marginals of $\nu$ are $\mu_n$ and $m$,
    and $T_{1,n}(x_1) = x_2$ for $\nu$-almost every $(x_1, x_2)$.
    Let $\nu' = (F_n)_* \mu'_n$.

    Since the marginals of $\nu$ and $\nu'$ on the second coordinate agree,
    by Shortt~\cite[Lemma~7]{S84} there exists a measure
    $\rho$ on $X \times X \times X$ with respective marginals
    $\mu_n$, $\mu'_n$ and $m$, such that $T_{1,n}(x_1) = T_{1,n}(x_2) = x_3$
    for $\rho$-almost every $(x_1, x_2, x_3)$.
    
    Let $\tmu_n$ be the marginal of $\rho$ on the first
    two coordinates. Then the marginals of $\tmu_n$ are $\mu_n$ and $\mu'_n$,
    and $\tmu_n(\Theta \leq n) = 1$.
    Now, $\tmu = \sum_{n \geq 1} \alpha_n \tmu_n$ is the required measure.
\end{proof}

\subsection{Moment bounds}

For a random variable $X$ we denote the strong and weak $L^p$ norms by
\begin{equation}
    \label{eq:Lp-norms}
    \|X\|_p
    = \bigl(\bE |X|^p \bigr)^{1/p}
    ,
    \qquad
    \|X\|_{p, \infty}
    = \bigl( \sup_{t > 0} t^p \bP(|X| > t) \bigr)^{1/p}
    .
\end{equation}
We note that $\|X\|_{p, \infty}$ is not, strictly speaking, a norm, but for $p > 1$
it is equivalent to the respective \emph{Lorentz norm,} which is indeed a
norm, see Stein and Weiss~\cite[Section~V.3]{SW71}.

We say that $H \colon X^\bN \to \bR$ is a separately Lipschitz function if
$\Lip_n(H) < \infty$ for each $n$, where
\[
    \Lip_n(H)
    = \sup_{\{x_k\}, x'_n} \frac{\bigl|
        H(x_0, \ldots, x_{n-1}, x_n, x_{n+1}, \ldots)
        -
        H(x_0, \ldots, x_{n-1}, x_n', x_{n+1}, \ldots)
    \bigr|}{d(x_n, x_n')}
    .
\]
Given a sequence of maps $T_1, T_2, \ldots \in \cT$,
slightly abusing notation where convenient, we use $H$ as a function of
a single variable:
\[
    H (x) = H(x, T_{1,1}(x), T_{1,2}(x), \ldots)
    .
\]
One example of a separately Lipschitz function is a Birkhoff sum
$H(x) = \sum_{k<n} v_k (T_{1,k}(x))$, as long as the observables
$v_k$ are Lipschitz. Another example is the running maximum
$H(x) = \max_{j\leq n} \bigl| \sum_{k<j} v_k (T_{1,k}(x)) \bigr|$.

The abstract version of Theorem~\ref{thm:main:moments} is:

\begin{thm}
    \label{thm:mom}
    Let $T_1, T_2, \ldots \in \cT$.
    Suppose that $h(n) = C_\beta n^{-\beta}$, where $C_\beta > 0$ and $\beta > 1$.
    Let $\mu$ be a regular probability measure on $X$ with tail bound $C_\beta n^{-\beta + 1}$.
    Let $H \colon X^\bN \to \bR$ be separately Lipschitz, continuous with respect to the product
    topology on $X^\bN$ and satisfying $\int H \, d\mu = 0$.
    Then on the probability space $(X, \mu)$:
    \begin{enumerate}[label=(\alph*)]
        \item If $\beta \in (1,2)$, then
            $ \displaystyle
                \| H \|_{\beta, \infty}
                \leq C \Bigl( \sum_{n \geq 0} \Lip_n(H)^\beta \Bigr)^{1/\beta}
                .
            $
        \item If $\beta = 2$, then
            $ \displaystyle
                \| H \|_2
                \leq C \Bigl( \sum_{n \geq 0} \Lip_n(H)^2 \bigl(1 + \log (n + 1)\bigr) \Bigr)^{1/2}
                .
            $
            In addition, for $p > 2$,
            \begin{align*}
                \| H \|_p
                & \leq C_p \Bigl( \sum_{n \geq 0} \Lip_n(H)^2 \bigl(1 + \log (n + 1)\bigr) \Bigr)^{1/2}
                \\
                & + C_p \Bigl( \sum_{n \geq 0} \Lip_n(H)^2 \Bigr)^{1/p}
                \Bigl( \sum_{n \geq 0} \Lip_n(H) \Bigr)^{1 - 2/p}
                .
            \end{align*}
        \item If $\beta > 2$, then
            $ \displaystyle
                \| H \|_{2(\beta - 1)}
                \leq C \Bigl( \sum_{n \geq 0} \Lip_n(H)^2 \Bigr)^{1/2}
                .
            $
    \end{enumerate}
    Here $C$ denotes constants which depend only on
    $C_\beta$, $\beta$, $K_2$, $K_1$ and $K$, $\diam X$, $\lambda$, $n_0$, $\delta_0$, 
    and $C_p$ depends also on $p$.
\end{thm}

Theorem~\ref{thm:mom} is proved in Sections~\ref{sec:mom:new} and~\ref{sec:mart}.
In the rest of this section  we show that Theorems~\ref{thm:main:memory} and~\ref{thm:main:moments}
fit our framework and follow from Theorems~\ref{thm:decdec} and~\ref{thm:mom}.

\subsection{Proof of Theorem~\ref{thm:main:memory}}

Fix $\gamma^* \in (0,1)$ and let $\cT$ be the family of intermittent maps~\eqref{eq:LSV}
with parameters in $(0, \gamma^*]$. Let $X = [0,1]$ and $Y = (1/2,1]$; let $m$ be the Lebesgue measure
on $Y$ normalized to probability and let $m_X$ be the Lebesgue measure on $X$.
Let $\beta = 1/ \gamma^*$. We use $C$ to denote various constants which depend only on $\gamma^*$.

Proposition~\ref{prop:LSV:NUE} verifies that $\cT$ satisfies the assumptions of
Section~\ref{sec:results} with the bound on return times $h(n) = C n^{-\beta}$.

\begin{prop}
    \label{prop:LSV:NUE}
    For each sequence $T_1, T_2, \ldots \in \cT$, there exists
    a partition $\cP_Y$ of $Y$ into intervals $(y_{n+1}, y_n]$ with $y_1 = 1$, $y_2 = 3/4$,
    $1/2 < y_{n+1} < y_n \leq 1/2 + C n^{-1/\gamma^*}$ and $y_n - y_{n+1} \leq y_{n+1} - 1/2$ for all $n$,
    such that $\tau \colon Y \to \{1,2,\ldots\}$,
    $\tau(y) = n$ if $y \in (y_{n+1}, y_n]$, is the first return time to $Y$.
    Further, each restriction $T_{1,n} \colon (y_{n+1}, y_n] \to Y$ is a bijection
    with bounded distortion:
    \[
        | \log T_{1,n}'(y) - \log T_{1,n}'(y') |
        \leq C |T_{1,n}(y) - T_{1,n}(y')|
        \quad \text{for all} \quad
        y,y' \in (y_{n+1}, y_n]
        .
    \]
\end{prop}

\begin{proof}
    Distortion bound is easily obtained from the Koebe principle, see e.g.\ \cite[Lemma 4.8]{BBD14}.
    For the bound $y_n - y_{n+1} \leq y_{n+1} - 1/2$ see~\cite[Equation~(4)]{L17}.
\end{proof}

There is a similar partition of $(0, 1/2]$:

\begin{prop}
    \label{prop:LSV:NUE:X}
    For each sequence $T_1, T_2, \ldots \in \cT$, there exists
    a partition $\cP_X$ of $(0, 1/2]$ into intervals $(x_{n+1}, x_n]$ with $x_1 = 1/2$,
    $x_{n+1} < x_n \leq C n^{-1/\gamma^*}$ and $x_n - x_{n+1} \leq x_{n+1}$ for all $n$,
    such that $\tau \colon (1/2,1] \to \{1,2,\ldots\}$,
    $\tau(x) = n$ if $x \in (x_{n+1}, x_n]$, is the first entry time to $Y$.
    Further, the restriction $T_{1,n} \colon (x_{n+1}, x_n] \to Y$ is a bijection
    with bounded distortion:
    \[
        | \log T_{1,n}'(x) - \log T_{1,n}'(x') |
        \leq C |T_{1,n}(x) - T_{1,n}(x')|
        \quad \text{for all} \quad
        x,x' \in (x_{n+1}, x_n]
        .
    \]
\end{prop}

Let $\mu$ be a probability measure on $X$ with Lipschitz density.
Then for sufficiently large $c > 0$, $\tmu = (\mu + c m_X) / (1 + c)$
is a probability measure with log-Lipschitz density and,
by Proposition~\ref{prop:LSV:NUE:X}, $\tmu$ is regular with tail bound $C n^{-\beta}$.
If $\mu'$ is another such measure, then
\[
    |(T_{1,n})_* (\mu - \mu')| / (1+c)
    = |(T_{1,n})_* (\tmu - \tmu')|
    = O(n^{-\beta})
\]
by Theorem~\ref{thm:decdec} and Remark~\ref{rmk:decdec}.
Measures with H\"older densities can be treated in the same way by Remark~\ref{rmk:Lipschitz}.
This proves Theorem~\ref{thm:main:memory}\ref{thm:main:memory:unexpected}.

Even though Theorem~\ref{thm:main:memory}\ref{thm:main:memory:normal} is an easier result,
its proof requires additional work:

\begin{prop}
    \label{prop:invreg}
    Let $\mu$ be a probability measure on $X$ with density in the cone $\cC_*$.
    Let $T_1, T_2, \ldots \in \cT$.
    Then $\mu$ is regular with tail bound $C n^{-\beta + 1}$.
    (For a suitable choice of $K_1$ in the definition of regularity.)
\end{prop}

\begin{proof}
    Let $\rho \in \cC_*$ be the density of $\mu$.
    Let $x_n$ and $y_n$ be as in Propositions~\ref{prop:LSV:NUE} and~\ref{prop:LSV:NUE:X}.

    The tail bound follows from $\rho(x) \leq C x^{-\gamma^*}$ and $x_n, y_n - 1/2 \leq C n^{-1/\gamma^*}$:
    \begin{align*}
        \mu \bigl( \{ x \in X : T_{1,k} \not \in Y \text{ for all } 1 \leq k < n\} \bigr)
        & = \mu \bigl((0, x_n] \cup (1/2, y_n] \bigr)
        \\
        & \leq C x_n^{-\gamma^* + 1} + C(y_n - 1/2)
        \leq C n^{1 - \beta}
        .
    \end{align*}

    It remains to show that $\mu$ is regular. Write
    \begin{align*}
        A_n 
        & = \{x \in X : T_{1,n}(x) \in Y \text{ \ and \ }
        T_{1,k}(x) \not \in Y \text{ for all } 1 \leq k < n \}
        \\
        & =  I_n \cup J_n,
    \end{align*}
    where $I_n = (x_{n+1}, x_n]$ and $J_n = (y_{n+1}, y_n ]$.
    We show that $\bigl|(T_{1,n})_*(\mu |_{I_n}) \bigr|_{\LL} \le C$ for all $n \ge 1$.
    The proof on $J_n$ is similar, the two together yield
    $\bigl| (T_{1,n})_*(\mu |_{A_n}) \bigr|_{\LL} \le C$ as wanted.

    The measure $(T_{1,n})_*(\mu |_{I_n})$ has density 
    \[
        \frac{ d (T_{1,n})_*(\mu |_{I_n})}{dm_X} (z)
        = \frac{ \rho(z_n) }{ T_{1,n}'(z_n) },
        \qquad z \in Y, 
    \]
    where $z_n =  (T_{1,n} |_{I_n})^{-1} z$.
    Hence it is enough to show that for all $z, z' \in Y$, 
    \begin{align}\label{eq:loglip-1}
        \biggl| \log \frac{ T_{1,n}'(z_n)  }{  T_{1,n}'(z'_n) } \biggr|
        \le C | z - z'|
    \end{align}
    and
    \begin{align}\label{eq:loglip-2}
        \biggl| \log \frac{ \rho(z_n) }{ \rho(z_n') } \biggr|
        \le C | z - z'|.
    \end{align}

    Inequality \eqref{eq:loglip-1} holds by Proposition~\ref{prop:LSV:NUE:X}.
    To obtain \eqref{eq:loglip-2} we assume that $z > z'$.
    Since $\rho$ is decreasing,
    \[
        \biggl|  \log \frac{ \rho(z_n)  }{  \rho(z_n') }  \biggr|
        = \log \frac{ \rho(z_n') }{\rho(z_n)}.
    \]
    Since $x^{\gamma^* + 1} \rho(x)$ is increasing,
    \begin{equation}
        \label{eq:gg}
        \frac{ \rho(z_n') }{\rho(z_n)}
        = \frac{  ( z_n')^{\gamma^* + 1} \rho(z_n') }{  z^{\gamma^* + 1}  \rho(z_n)}
        \frac{ z_n^{\gamma^* + 1} }{ ( z_n')^{\gamma^* + 1}} 
        \le \frac{ z_n^{\gamma^* + 1} }{ ( z_n')^{\gamma^* + 1}}.
    \end{equation}
    By the distortion bound~\eqref{eq:loglip-1}, $T_{1,n}' \geq C (x_n - x_{n+1})^{-1}$ on $(x_{n+1}, x_n]$.
    Hence
    \begin{equation}
        \label{eq:ggg}
        z_n - z_n'
        \leq \inf_{x \in (x_{n+1}, x_n]} T_{1,n}'(x) \; (z - z')
        \leq C (x_n - x_{n+1}) (z - z')
        .
    \end{equation}
    By~\eqref{eq:gg} and~\eqref{eq:ggg}, and using $z_n \in (x_{n+1}, x_n]$,
    \[
        \log \frac{ \rho(z_n') }{\rho(z_n)}
        \leq C( \log z_n - \log z_n' )
        \leq x_{n+1}^{-1} (z_n - z_n')
        \leq C x_{n+1}^{-1} (x_n - x_{n+1}) (z - z')
        .
    \]
    Now to obtain \eqref{eq:loglip-2} it suffices to recall that
    $x_n - x_{n+1} \leq x_{n+1}$.
\end{proof}
Theorem~\ref{thm:main:memory}\ref{thm:main:memory:normal} follows from
Theorem~\ref{thm:decdec}, Theorem~\ref{thm:main:memory}\ref{thm:main:memory:unexpected}
and Proposition~\ref{prop:invreg}.

\subsection{Proof of Theorem~\ref{thm:main:moments}}

We showed above that the maps in $\cT$ satisfy assumptions of Theorem~\ref{thm:mom}
and that every probability measure with density in $\cC_*$ is regular
with tail bound $C n^{-1/\gamma^* + 1}$.
The bounds in Theorem~\ref{thm:main:moments} follow from those in
Theorem~\ref{thm:mom} and $|S_n| \leq n \sup_{k,x} |v_k(x)|$,
as in~\cite[Equations~(1.2),~(1.3)]{GM14}.


\section{Proof of Theorem~\ref{thm:decdec}}
\label{sec:mixing}

Our proof of Theorem~\ref{thm:decdec} is not very long but technical
and the idea is hard explain in simple terms.
(Which is, in a way, similar to Lindvall's extremely short but puzzling
proof for homogeneous Markov chains~\cite{Li79}.)
Informally, the main steps are:
\begin{enumerate}
    \item \emph{Recurrence to $Y$, Proposition~\ref{prop:onedec}.}
        Given a ``nice'' probability measure $\mu$ on $Y$, we show
        that for all sufficiently large $n$,
        \[
            (T_{1,n})_* \mu = \theta m + (1-\theta) \mu'
            ,
        \]
        where $\theta \in (0,1)$ is independent of $\mu$ and $n$, and
        $\mu'$ has tail bound $h_n \sim \sum_{j=0}^n h(j + \ell)$.
        The tail bounds $h_n$ grow with $n$, but are just enough for the rest of our proof
        to work.
    \item \emph{Returns to $Y$ for regular measures, Lemma~\ref{lem:mixdec}.}
        We expand the previous step to show that for a regular measure $\mu$ on $X$,
        \[
            \mu = \sum_{j=1}^{\infty} \alpha_j [\theta \mu'_j + (1-\theta) \mu_j ]
            ,
        \]
        where the constants $\alpha_j \in [0,1]$ are fully and explicitly determined
        by the tail of $\mu$, and $\mu_j$, $\mu'_j$ are probability
        measures with $(T_{1,j})_* \mu_j$ regular with tail bound $h_j$
        and $(T_{1,j})_* \mu'_j = m$.
    \item \emph{Reduction to a probabilistic problem, Lemma~\ref{lem:proproprobab}.}
        Iterating the previous step (successively applying it to measures $\mu_j$
        and to the results of their decompositions), we decompose a general regular measure
        $\mu$ with tail $r$ into
        \[
            \mu = \sum_{n=1}^\infty \bP(S = n) \mu_n
            ,
        \]
        where $(T_{1,n})_* \mu_n = m$ and $S$ is a random variable, constructed on
        an unrelated probability space as $S = X_1 + \ldots + X_\tau$.
        Here $\tau$ is a geometric random variable with parameter $\theta$
        and $X_n$ are (independent from $\tau$) random variables with explicitly controlled tails:
        \begin{align*}
            \bP(X_1 \geq \ell) 
            & \sim r(\ell - n_0)
            \\
            \bP(X_j \geq \ell \mid X_1, \ldots, X_{j-1}) 
            & \sim h_{X_{j-1}}(\ell - n_0)
            \quad \text{ for } j \geq 2
            .
        \end{align*}
    \item \emph{Tail estimates, Propositions~\ref{prop:Stail},~\ref{prop:Stail-exp}.}
        Finally we estimate $\bP(S \geq n)$ for specific bounds on $r$ and $h$.
\end{enumerate}

We begin the proof with a simple yet important observation:

\begin{rmk}
    \label{rmk:on-Y}
    Using the inequality
    $(a+b) / (a' + b') \leq \max \{a/a', b/b'\}$, which holds for
    $a,a',b,b' > 0$, from the definition of $|\cdot|_{\LL}$
    we deduce that if $\mu$ and $\mu'$ are nonnegative measures on $Y$, then
    \[
        |\mu + \mu'|_{\LL} 
        \leq \max \{ |\mu|_{\LL}, |\mu'|_{\LL} \}
        ,
    \]
    whenever the above is well defined.
    This inequality extends to finite and countable sums:
    $|\sum_k \mu_k |_{\LL} \leq \sup_k |\mu_k|_{\LL}$.
    As a corollary, if $\mu$ is a measure on $Y$ with $|\mu|_{\LL} \leq K_2$, or
    more generally a regular measure on $X$, then for each $n \geq 1$,
    \[
        \Bigl| \bigl( (T_{1,n})_* \mu \bigr) \big|_Y \Bigr|_{\LL}
        \leq K_1
        .
    \]
\end{rmk}

\begin{prop}
    \label{prop:onedec}
    Suppose that $\mu$ is a probability measure on $Y$ with $|\mu|_{\LL} \leq K_2$.
    Let $T_1, T_2, \ldots \in \cT$.
    Let $h_n (\ell) = C_h \sum_{j=0}^n h(j + \ell)$, where $C_h = 2 e^{K_2 \diam Y}$.
    Then:
    \begin{enumerate}[label=(\alph*)]
        \item\label{prop:onedec:tail}
            For every $n \geq 0$ the tail of $(T_{1,n})_* \mu$ is bounded by $\frac{1}{2} h_n(\ell)$.
        \item\label{prop:onedec:dec}
            There is a constant $\theta \in (0,1)$, depending only on $K_1$, $K_2$, $\diam Y$ and $\delta_0$,
            such that for every $n \geq n_0$,
            \[
                (T_{1,n})_* \mu
                = \theta m + (1-\theta) \mu'
                ,
            \]
            where $\mu'$ is a regular probability measure.
            The tail of $\mu'$ is bounded by
            $h_n (\ell) = C_h \sum_{j=0}^n h(j + \ell)$, where $C_h = 2 e^{K_2 \diam Y}$.
    \end{enumerate}
\end{prop}

\begin{proof}
    We prove \ref{prop:onedec:tail} first.
    Suppose that $n \geq 0$. For each $0 \le j \le n$ define
    \begin{align*}
        Y_j &
        = \{ y \in Y : T_{n-j + 1, \ell }(y) \notin Y  \text{ for all $n-j <  \ell \le n$}  \}
        ,
        \\
        Y_j' 
        &= Y \cap   T_{1, n-j}^{-1}   Y_j 
        \\
        & = \{ y \in Y : T_{1, n-j}(y) \in Y \text{ and $T_{1,\ell}(y) \notin Y$ for all $n - j < \ell \le n$} \}.
    \end{align*}
    Observe that the sets $Y_j'$ form a partition of $Y$, so we can write
    \[
        ( T_{1,n}  )_*\mu = \sum_{j=0}^n ( T_{1,n}  )_*\mu_j,
    \]
    where $\mu_j$ is the restriction of $\mu$ to $Y_j'$.

    Next, set
    $\nu_j = (( T_{1, n-j} )_*\mu)|_Y$ for all $0 \le j \le n$ and note that for all measurable $B \subset X$,
    \begin{align*}
        (  T_{n-j+1, n} )_* ( \nu_j |_{Y_j}  ) (B) = \nu_j( Y_j \cap T_{n-j+1, n} ^{-1} (B) )
        = \mu(   T_{1, n-j}^{-1} Y_j   \cap T_{1,n}^{-1}B  ) = (T_{1,n} )_*\mu_j(B),
    \end{align*}
    in other words $(T_{n-j+1, n} )_* ( \nu_j |_{Y_j} ) = (T_{1,n} )_*\mu_j$.

    By Remark~\ref{rmk:on-Y}, $|\nu_j|_{\LL} \leq K_2$; using $\nu_j(Y) \leq 1$ we deduce that
    $\nu_j \leq e^{K_2 \diam Y} m$ and thus the tail of $\nu_j$ is bounded
    by $e^{K_2 \diam Y} h$. Observe that
    $(T_{n-j+1, n} )_* ( \nu_j |_{Y_j} )$ inherits the tail bound from $\nu_j$ with a time shift, namely
    $(T_{n-j+1, n} )_* ( \nu_j |_{Y_j} )$  has tail bound $e^{K_2 \diam Y} h(\cdot + j)$.
    It follows that $(T_{1,n})_* \mu$ has tail bound
    $e^{K_2 \diam Y} \sum_{j=0}^n h(\cdot + j)$, as required.
    
    It remains to prove~\ref{prop:onedec:dec}.
    Let $\theta_0 \in (0,1)$ be such that for every $\theta' \in [0,\theta_0]$,
    every measure $\rho$ on $Y$ with $|\rho|_{\LL} \leq K_1$  can be written as
    $\rho = \rho(Y) \theta' m + \rho'$ with $|\rho'|_{\LL} \leq K_2$.
    Such $\theta_0$ exists and only depends on $K_1, K_2$ and $\diam Y$, see~\cite[Lemma~3.4]{KKM19}.

    Suppose that $n \geq n_0$ and let $\rho_n = \bigl((T_{1,n})_* \mu\bigr)\big|_Y$.
    Observe that $\bigl| \rho_n \bigr|_{\LL} \leq K_1$ and $\rho_n(Y) \geq \delta_0$.
    Let $\theta = \min \{\theta_0 \delta_0, 1/2\}$.
    Then
    \[
        \rho_n = \theta m + \rho'
        \qquad \text{with } |\rho'|_{\LL} \leq K_2
        .
    \]
    Define
    \[
        \mu'
        = (1-\theta)^{-1} \bigl( (T_{1,n})_* \mu - \theta m \bigr)
        = (1-\theta)^{-1} \Bigl( \rho' + \bigl((T_{1,n})_* \mu \bigr)\big|_{X \setminus Y} \Bigr)
        .
    \]
    Then $\mu'$ is a probability measure and
    $(T_{1,n})_* \mu = \theta m + (1-\theta) \mu'$.
    Both $\rho'$ and $\bigl((T_{1,n})_* \mu \bigr)\big|_{X \setminus Y}$
    are regular measures, and thus so is $\mu'$. To bound the tail of
    $\mu'$, we note that
    $\mu' \leq (1-\theta)^{-1} (T_{1,n})_* \mu$ with $(1-\theta)^{-1} \leq 2$,
    and apply the bound on the tail of $(T_{1,n})_* \mu$.

\end{proof}

Further we use $\theta$, $h_n$ and $C_h$ from Proposition~\ref{prop:onedec}.

\begin{cor}
    \label{cor:onedec}
    Let $T_1, T_2, \ldots \in \cT$.
    Suppose that $N \geq 0$ and $\mu$ is a probability measure on $X$ such that $(T_{1,N})_* \mu$
    is supported on $Y$ and $|(T_{1,N})_* \mu|_{\LL} \leq K_2$.
    Then for every $n \geq N + n_0$,
    \[
        \mu = \theta \mu_n + (1-\theta) \mu'_n
        ,
    \]
    where $\mu_n$, $\mu'_n$ are probability measures with $(T_{1,n})_* \mu_n = m$
    and $(T_{1,n})_* \mu'$ regular with tail bound $h_{n-N}$.
\end{cor}

\begin{proof}
    Fix $n \geq N + n_0$.
    Proposition~\ref{prop:onedec} gives the decomposition 
    $(T_{1,n})_* \mu = \theta m + (1-\theta) \mu'$,
    where $\mu'$ is a regular probability measure with tail bound $h_{n-N}$.
    Define $\mu_n$ and $\mu_n'$ by
    \[
        d \mu_n = \Bigl( \frac{dm}{d(T_{1,n})_* \mu} \circ T_{1,n} \Bigr) \, d \mu
        \qquad \text{and} \qquad
        d \mu'_n = \Bigl( \frac{d\mu'}{d(T_{1,n})_* \mu} \circ T_{1,n} \Bigr) \, d \mu
        .
    \]
    It is straightforward that $\mu = \theta \mu_n + (1-\theta) \mu'_n$ with
    $(T_{1,n})_*\mu_n = m$ and $(T_{1,n})_*\mu'_n = \mu'$.
    This is the desired decomposition.
\end{proof}

\begin{cor}
    \label{cor:const-tail}
    Let $T_1, T_2, \ldots \in \cT$ and let $\mu$ be a regular probability measure
    with tail bound $r$.
    Then $(T_{1,n})_* \mu$ is a regular probability measure with
    tail bound $r_n(\ell) = r(n+\ell) + C_h \sum_{j=0}^n h(j + \ell)$.
\end{cor}

\begin{proof}
    Let $\cP$ be a partition of $X$ corresponding to $T_1, T_2, \ldots$
    and let $a \in \cP$. Let $\mu_a$ be the restriction
    of $\mu$ on $a$. By Proposition~\ref{prop:onedec}, if $\tau(a) \leq n$ then
    $(T_{1,n})_* \mu_a$ has tail bound $\ell \mapsto |\mu_a| C_h \sum_{j=0}^n h(j + \ell)$.
    Thus $(T_{1,n})_* \bigl( \sum_{a \in \cP : \tau(a) \leq n} \mu_a \bigr)$ has tail bound
    $\ell \mapsto C_h \sum_{j=0}^n h(j + \ell)$.
    It remains to notice that $(T_{1,n})_* \bigl( \mu - \sum_{a \in \cP : \tau(a) \leq n} \mu_a \bigr)$
    has tail bound $\ell \mapsto r(n + \ell)$.
\end{proof}

\begin{lem}
    \label{lem:mixdec}
    Suppose that $\mu$ is a regular probability measure with tail bound $r$ where 
    $r(n)$ is nondecreasing, $r(1) = 1$ and $\lim_{n \to \infty} r(n) = 0$.
    Suppose that $T_1, T_2, \ldots \in \cT$.
    Then
    \[
        \mu = \sum_{j=n_0+1}^{\infty} \alpha_j [\theta \mu'_j + (1-\theta) \mu_j ]
        ,
    \]
    where $\alpha_j = r(j-n_0) - r(j+1-n_0)$ and $\mu_j$, $\mu'_j$ are probability measures such that
    $(T_{1,j})_* \mu_j$ is regular with tail bound $h_j$
    and $(T_{1,j})_* \mu'_j = m$.
\end{lem}

\begin{proof}
    Let $\cP$ be the partition of $X$ corresponding to $T_1, T_2, \ldots$ and let
    \[
        A_n
        = \cup \{ a \in \cP : \tau(a) = n \}
        .
    \]
    Let $\nu_n = \mu|_{A_n}$.   
    Then for each $n \geq 1$, the measure $(T_{1,n})_* \nu_n$ is supported on $Y$ and
    satisfies $|(T_{1,n})_* \nu_n|_{\LL} \leq K_1$.
    By Corollary~\ref{cor:onedec}, for each $\ell \geq n + n_0$,
    \begin{equation}
        \label{eq:gggn}
        \nu_n = \mu(A_n) \bigl[ \theta \nu_{n, \ell} + (1-\theta) \nu_{n,\ell}' \bigr]
        ,
    \end{equation}    
    where $\nu_{n,\ell}$, $\nu_{n,\ell}'$ are probability measures with
    $(T_{1,\ell})_* \nu_{n,\ell} = m$ and $(T_{1,\ell})_* \nu_{n,\ell}'$ regular with tail bound $h_{\ell-n}$.

    We observe that
    \[
        \sum_{n=1}^\ell \mu(A_n) \geq r(1) - r(\ell)
        \quad \text{ and } \quad
        \sum_{n=1}^\infty \mu(A_n) = r(1) - \lim_{\ell \to \infty} r(\ell) = 1
        .
    \]
    Hence (see \cite[Proposition~4.7]{KKM19}) there exist nonnegative constants
    $\xi_{\ell,n}$, $1 \leq n \leq \ell < \infty$, such that
    \begin{align*}
        \sum_{n \leq \ell} \xi_{\ell,n} \mu(A_n)
        & = r(\ell) - r(\ell+1) 
        & & \text{ for each } \ell
        \\
        \sum_{\ell \geq n} \xi_{\ell,n} 
        &= 1
        & & \text{ for each } n
        .
    \end{align*}
    
    For $j \geq n_0 + 1$, let
    $\chi_j = \sum_{n=1}^{j - n_0} \xi_{j-n_0,n} \nu_n$. Then $\mu = \sum_{j=n_0+1}^\infty \chi_j$ and
    $\chi_j(X) = \alpha_j$. Due to~\eqref{eq:gggn},
    \[
        \chi_j = \alpha_j [ \theta \mu_j + (1-\theta) \mu'_j ]
    \]
    with $\mu_j  = \alpha_j^{-1} \sum_{n=1}^{j-n_0} \xi_{j-n_0, n} \mu(A_n) \nu_{n,j}$
    and  $\mu'_j = \alpha_j^{-1} \sum_{n=1}^{j-n_0} \xi_{j-n_0, n} \mu(A_n) \nu'_{n,j}$.
    (It is possible that $\alpha_j = 0$, but this does not create problems and we ignore it
    for simplicity.)

    It remains to observe that $\mu_j$ and $\mu'_j$ are probability measures with $(T_{1,j})_* \mu_j = m$
    and $(T_{1,j})_* \mu'_j$ regular with tail bound $h_j$.
\end{proof}

Similar to Corollary~\ref{cor:onedec} we obtain:

\begin{cor}
    \label{cor:mixdec}
    Suppose that $T_1, T_2, \ldots \in \cT$.
    Suppose that $N \geq 0$ and $\mu$ is a probability measure such that
    $(T_{1,N})_* \mu$ is regular has tail bound $r$ where $r(n)$ is nondecreasing, $r(1) = 1$
    and $\lim_{n \to \infty} r(n) = 0$.
    Then
    \[
        \mu = \sum_{j=n_0+1}^{\infty} \alpha_j [\theta \mu'_j + (1-\theta) \mu_j ]
        ,
    \]
    where $\alpha_j = r(j-n_0) - r(j+1-n_0)$ and $\mu_j$, $\mu'_j$ are probability measures such that
    $(T_{1,N+j})_* \mu_j$ is regular with tail bound $h_j$
    and $(T_{1,N+j})_* \mu'_j = m$.
\end{cor}

Further we suppose that $r$ is nonnegative with $\lim_{n \to \infty} r(n) = 0$
and define
\[
    \hr (n) = \min \{ 1, r(1), \ldots, r(n)\}
    .
\]
This way, $\hr$ is nonincreasing and $\hr(1) = 1$;
for a probability measure, tail bound $r$ is equivalent to $\hr$.
Similarly define $\hh$ and $\hh_k$.

Let $X_1, X_2, \ldots$ be random variables with values in $\{n_0, n_0+1,\ldots\}$ such that
for all $\ell \geq n_0$,
\begin{align*}
    \bP(X_1 \geq \ell) 
    &= \hr(\ell - n_0)
    \\
    \bP(X_j \geq \ell \mid X_1, \ldots, X_{j-1}) 
    & = \hh_{X_{j-1}}(\ell - n_0)
    \quad \text{ for } j \geq 2
    .
\end{align*}
Let $\tau$ be a geometric random variable on $\{1,2,\ldots\}$ with parameter $\theta$,
namely $\bP(\tau = \ell) = (1-\theta)^{\ell-1} \theta$. Let $\tau$ be
independent from $\{X_j\}$.
Let
\[
    S 
    = X_1 + \ldots + X_\tau
    .
\]

\begin{lem}
    \label{lem:proproprobab}
    Suppose that $\mu$ is as in Theorem~\ref{thm:decdec}
    and $T_1, T_2, \ldots \in \cT$.
    Then there exists a decomposition
    \[
        \mu = \sum_{n=1}^\infty \bP(S = n) \mu_n
    \]
    where $\mu_n$ are probability measures such that $(T_{1,n})_* \mu_n = m$.
\end{lem}

\begin{proof}
    Starting from the decomposition from Corollary~\ref{cor:mixdec},
    we apply the same decomposition to $\mu_j$ and so on recursively to obtain:
    \begin{equation}
        \label{eq:hnak}
        \mu = \theta \sum_{j > n_0} \alpha_j \mu'_j
        + (1-\theta) \theta \sum_{j,k > n_0} \alpha_{j,k} \mu'_{j,k}
        + (1-\theta)^2 \theta \sum_{j,k,\ell > n_0} \alpha_{j,k,\ell} \mu'_{j,k,\ell}
        + \cdots
    \end{equation}
    where
    \begin{itemize}
        \item $(T_{1,j})_* \mu'_j = m$ and $\alpha_j = \hr(j-n_0) - \hr(j+1-n_0)$,
        \item $(T_{1,j+k})_* \mu'_{j,k} = m$ and $\alpha_{j,k} = \alpha_j (\hh_j(k-n_0) - \hh_j(k+1-n_0))$,
        \item $(T_{1,j+k+\ell})_* \mu'_{j,k,\ell} = m$ and
            $\alpha_{j,k,\ell} = \alpha_{j,k} (\hh_k(\ell-n_0) - \hh_k(\ell+1-n_0))$,
        \item[] and so on.
    \end{itemize}

    We observe that for each $n \geq 1$ and $j_1, j_2, \ldots, j_n \geq n_0$,
    \[
        (1-\theta)^{n-1} \theta \alpha_{j_1, j_2, \ldots, j_n}
        = \bP(\tau = n, X_1 = j_1, \ldots, X_n = j_n)
        .
    \]
    Grouping the terms in~\eqref{eq:hnak} by the sum of indices,
    we obtain the required decomposition with
    \[
        \mu_n =
        \sum_{\substack{k \geq 1 \\  j_1+\cdots+j_k=n}}
        \alpha_{j_1, \ldots, j_k} \mu'_{j_1, \ldots, j_k}
        \Bigg/
        \sum_{\substack{k \geq 1 \\  j_1+\cdots+j_k=n}}
        \alpha_{j_1, \ldots, j_k}
        .
    \]
\end{proof}

To complete the proof of Theorem~\ref{thm:decdec}, it remains to estimate the tails
$\bP(S \geq n)$, as it is done in the following two propositions.

\begin{prop}
    \label{prop:Stail}
    Suppose that $\mu$ is as in Theorem~\ref{thm:decdec}, $T_1, T_2, \ldots \in \cT$
    and $h(n) \leq C_\beta n^{-\beta}$ with $\beta > 1$.
    \begin{enumerate}[label=(\alph*)]
        \item\label{prop:Stail:pol}
            If $r(n) \leq C'_\beta n^{-\beta'}$ where $\beta' \in (0, \beta]$, then
            \[
                \bP(S \geq n)
                \leq C'_\beta C n^{-\beta'}
                ,
            \]
            where $C$ depends only on $n_0$, $\theta$, $C_h$, $C_\beta$, $\beta'$ and $\beta$.
        \item\label{prop:Stail:exp}
            If $\sum_{j=1}^\infty r(j) < \infty$ (which corresponds to $\bE X_1 < \infty$), then
            for $n \geq 2 n_0$,
            \[
                \bP(S \geq n)
                \leq r(\floor{n / 2} - n_0) + C n^{-\beta} \sum_{j=1}^\infty r(j)
                ,
            \]
            where $C$ depends only on $n_0$, $\theta$, $C_h$, $C_\beta$ and $\beta$.
    \end{enumerate}
\end{prop}

\begin{proof}
    Let $C$ denote various constants which depend only on $n_0$, $\theta$, $C_h$, $C_\beta$ and $\beta$.
    Suppose, without loss of generality, that $h$ is nonincreasing, so that $\hh_n(\ell) \leq h_n(\ell)$.
    
    Write, for $k \geq 2$,
    \begin{align*}
        \bE(X_k \mid X_{k-1}) - n_0
        & = \sum_{j=n_0 + 1}^\infty \bP(X_k \geq j \mid X_{k-1})
        = \sum_{j=n_0 + 1}^\infty \hh_{X_{k-1}}(j - n_0)
        \\
        & \leq \sum_{j=n_0 + 1}^\infty h_{X_{k-1}}(j - n_0)
        = C_h \sum_{\ell = 0}^{X_{k-1}} \sum_{j=n_0 + 1}^\infty h(\ell + j - n_0)
        \\
        & \leq C_h C_\beta \sum_{\ell = 0}^{X_{k-1}} \sum_{j=n_0 + 1}^\infty (\ell + j - n_0)^{-\beta}
        \leq C + X_{k-1} / 2
        .
    \end{align*}
    Hence $\bE(X_k \mid X_1) \leq C + \bE(X_{k-1} \mid X_1) / 2$.
    By induction,
    \begin{equation}
        \label{eq:ltrrr}
        \sup_{k \geq 2} \bE (X_k \mid X_1)
        \leq C X_1
        .
    \end{equation}

    Next, for $k \geq 2$ and $j > n_0$,
    \begin{align*}
        \bP(X_k \geq j \mid X_{k-1})
        & = \hh_{X_{k-1}}(j - n_0)
        \leq C_h C_\beta (X_{k-1} + 1) (j-n_0)^{-\beta}
        .
    \end{align*}
    Taking conditional expectation with respect to $X_1$ and using~\eqref{eq:ltrrr},
    we obtain
    \begin{equation}
        \label{eq:Xtail}
        \bP(X_k \geq j \mid X_1)
        \leq C j^{-\beta} X_1
        \qquad \text{for all } k \geq 2 \text{ and } j \geq 1
        .
    \end{equation}

    We prove~\ref{prop:Stail:pol} first.
    By~\eqref{eq:Xtail}, using $\beta' \in (0, \beta]$ and $\beta > 1$,
    \begin{align*}
        \bP(X_k \geq j)
        & \leq \bE \min\{C j^{-\beta} X_1, 1\}
        \leq C j^{-\beta} \sum_{\ell=1}^{C j^\beta} \bP(X_1 \geq \ell)
        \\
        & \leq C C'_\beta j^{-\beta} \sum_{\ell=1}^{C j^\beta} \ell^{-\beta'}
        \leq C C'_\beta j^{-\beta}
        \begin{cases}
            1, & \beta' > 1
            \\
            \log j, & \beta' = 1
            \\
            j^{\beta(-\beta' + 1)}, & \beta' < 1
        \end{cases}
        \\
        & \leq C C'_\beta j^{-\beta'}
        .
    \end{align*}
    Thus $\bP(X_k \geq j) \leq C C'_\beta j^{-\beta'}$ for all $k \geq 1$ and $j \geq 1$.
    Hence
    \begin{equation}
        \label{eq:777}
        \begin{aligned}
            \bP(S \geq n)
            & = \sum_{t=1}^{\infty} \bP(\tau = t) \bP(X_1 + \cdots + X_t \geq n)
            \\ 
            & \leq \sum_{t=1}^{\infty} (1-\theta)^{t-1} \theta 
            \bigl[ \bP(X_1 \geq n/t) + \cdots + \bP(X_t \geq n/t) \bigr]
            \\ 
            & \leq C C'_\beta \sum_{t=1}^{\infty} (1-\theta)^{t-1} \theta
            t^{1+\beta'} n^{-\beta'}
            \leq C C'_\beta n^{-\beta'}
            ,
        \end{aligned}
    \end{equation}
    as required.
    
    Now we prove~\ref{prop:Stail:exp}.
    Let $C_r = \sum_{\ell=1}^\infty r(\ell)$; note that $\bP(X_1 \geq j) \leq r(j - n_0)$ for $j \geq n_0$,
    and that $\bE X_1 \leq C_r$.
    Taking expectation of both sides in~\eqref{eq:Xtail},
    we obtain
    \[
        \bP(X_k \geq j)
        \leq C j^{-\beta} \bE X_1
        \leq C C_r j^{-\beta}
        \qquad \text{for all } k \geq 2
        .
    \]
    Similar to~\eqref{eq:777}, we have
    $\bP(S - X_1 \geq n) \leq C C_r n^{-\beta}$.
    The result follows from $\bP(S \geq n) \leq \bP(X_1 \geq n/2) + \bP(S - X_1 \geq n/2)$.
\end{proof}

\begin{cor}
    \label{cor:Stail}
    Suppose that $T_1, T_2, \ldots \in \cT$ and $h(\ell) \leq C_\beta \ell^{-\beta}$
    with $\beta > 1$. Let $\mu$ and $\nu$ be probability measures
    on $Y$ with $|\mu|_{\LL}, |\nu|_{\LL} \leq K_2$. Then
    \[
        \bigl| (T_{1,k+n})_* \mu - (T_{k,k+n})_* \nu \bigr|
        \leq C \min \{ k n^{-\beta}, n^{-\beta + 1} \}
        ,
    \]
    where $C$ is a constant which depends only on $C_\beta$, $\beta$, $K_2$, $K_1$
    and $K$, $\diam X$, $\lambda$, $n_0$, $\delta_0$.
\end{cor}

\begin{proof}
    By Proposition~\ref{prop:onedec}, $(T_{1,k})_* m$ is a regular measure with
    tail bound $h_k$. It is a direct verification that
    $h_k(\ell) \leq C \ell^{-\beta + 1}$ and
    $h_k(\ell) \leq C k \ell^{-\beta}$.
    Now apply Proposition~\ref{prop:Stail}.
\end{proof}

One last thing for us to prove is the bound for (stretched) exponential tails.

\begin{prop}
    \label{prop:Stail-exp}
    Suppose that $\mu$ is as in Theorem~\ref{thm:decdec}, $T_1, T_2, \ldots \in \cT$
    and that $h(n), r(n) \leq C_\beta \exp( - C_\beta' n^{\beta} )$ with $\beta \in (0,1]$
    and $C_\beta, C_\beta' > 0$.
    Then
    \[
        \bP(S \geq n)
        \leq C \exp( - C' n^{\beta} )
        ,
    \]
    where $C > 0$ and $C' \in (0, C_\beta')$ depend only on
    $n_0$, $\theta$, $C_h$, $C_\beta$, $C_\beta'$ and $\beta$.
\end{prop}

\begin{proof}
    We give a sketch of the proof.   
    We use $C, C'$ to denote various constants which, as in the statement,
    depend only on $n_0$, $\theta$, $C_h$, $C_\beta$, $C_\beta'$ and $\beta$.
    Note that 
    \[
        h_k(\ell)
        \leq C \sum_{j=0}^{\infty} \exp( - C_\beta' (\ell + j)^{\beta} )
        \leq C \exp( - C' \ell^{\beta} )
        ,
    \]
    and thus
    \[
        \bP(X_k \geq \ell \mid X_1, \ldots, X_{k-1})
        \leq C \exp( - C' \ell^{\beta} )
        .
    \]
    Now the result follows as in~\cite[Propositions~4.11 and~4.12]{KKM19}.
\end{proof}


\section{Proof of Theorem~\ref{thm:mom}}
\label{sec:mom:new}

Throughout this section, $C$ denotes various constants which depend only on
$\beta$, $C_\beta$, $\lambda$, $K$, $K_1$, $K_2$, $\delta_0$, $n_0$ and $\diam X$.
We work on the probability space $(X, \mu)$, and $\bE$ denotes the expectation
with respect to $\mu$.

Overall,
our strategy is to construct a filtration $\cB_n$ (based on symbolic itinerary),
approximate $H$ with the Doob martingale $\tH_n = \bE(H \mid \cB_n)$. Then we bound
the quadratic variation of $\tH_n$ and use Burkholder inequality.

\subsection{Filtration and martingale}
For $n \geq 0$, let $\cP_n$ denote the partition of $X$ corresponding to the sequence of maps
$T_{n+1}, T_{n+2}, \ldots$ as in Section~\ref{sec:results}.
To each $x \in X$ there corresponds a symbolic
itinerary $a_0, a_1, \ldots$ with $a_n \in \cP_n$ and $T_{1,n}(x) \in a_n$.
Let $\cB_n$ denote the sigma-algebra generated by $a_0, \ldots, a_n$, i.e.\ by sets of the type
$\{ x \in X : T_{1,k}(x) \in a_k \text{ for } 0 \leq k \leq n \}$.
Let $\cB_{-1} = \{\emptyset, X\}$ be the trivial sigma-algebra.

Let
\[
    \tH_n
    = \bE(H \mid \cB_n)
    .
\]
Then $\tH_n$ is a (Doob) martingale.
Note that $\tH_{-1} = 0$.
Let $\tX_{-1} = 0$ and for $n \geq 0$,
\[
    \tX_n = \tH_n - \tH_{n-1}
    .
\]

\begin{rmk}
    In Theorem~\ref{thm:mom} we assumed that $H$ is continuous on $X^\bN$ with
    respect to the product topology. Since returns to $Y$ are backward contracting,
    this guarantees that $\tH_n \to H$ pointwise.
\end{rmk}

To estimate the increments $\tX_n$ we use some auxiliary random variables.
For $x \in X$, define the sequence of return times to $Y$ by
$r_{-1}(x) = 0$ and
\begin{align*}
    r_k(x)
    = \inf \{ \ell > r_{k-1}(x) \: : \: T_{1,\ell}(x) \in Y \}
    \quad \text{for } k \ge 0
    .
\end{align*}
Define lap numbers
\begin{align*}
    L_k(x) = \# \{ 1 \leq \ell \leq k \: : \: T_{1,\ell}(x) \in Y \}. 
\end{align*}
Then $r_{L_k - 1} \leq k < r_{L_k}$.
Observe that $L_k$ and $r_{L_k}$ are $\cB_k$-measurable.
Denote
\begin{align*}
    \varkappa_k & = r_{L_k}
    , \\
    \tau_k & = r_k - r_{k-1}
    , \\
    \htau_k & = \sum_{r_{k-1} \leq j < r_k} \Lip_j (H)
    .
\end{align*}

\subsection{Martingale increments}

Throughout this subsection, we fix a symbolic itinerary $a_0, a_1, \ldots$
and let $A_{-1} = X$ and for $n \geq 0$,
\[
    A_n
    = \{ x \in X : T_{1,k}(x) \in a_k \text{ for all } k \leq n \}
    .
\]

\begin{prop}
    \label{prop:back-cont}
    For all $x,x' \in A_{n-1}$,
    \[
        \sum_{k < \varkappa_n(A_n)} \Lip_k (H) d(T_{1,k}(x), T_{1,k}(x'))
        \leq C \sum_{\ell \leq L_n(A_n)} \htau_{\ell}(A_n) \lambda^{-(L_n(A_n) - \ell)}
        .
    \]
\end{prop}

\begin{proof}
    Suppose that $x,x' \in A_{n-1}$.
    By backward contraction of at least $\lambda$ at returns to $Y$
    and using $L_n(A_n) \leq L_{n-1}(A_{n-1}) + 1$
    and $L_j(A_n) = L_j(A_{n-1})$ for $j < n$,
    \[
        d(T_{1,j}(x), T_{1,j}(x'))
        \leq C \lambda^{-(L_n(A_n) - L_j(A_n))}
        \quad \text{for } j < \varkappa_n(A_n)
        .
    \]
    Hence
    \begin{align*}
        \sum_{j < \varkappa_n(A_n)} \Lip_j(H) d(T_{1,j}(x), T_{1,j}(x'))
         & = \sum_{\ell=0}^{L_n(A_n)} \sum_{j=r_{\ell-1}}^{r_\ell(A_n) - 1}
         \Lip_j(H) d(T_{1,j}(x), T_{1,j}(x'))
         \\
         & \leq \sum_{\ell=0}^{L_n(A_n)} \htau_{\ell} (A_n) \lambda^{- (L_n(A_n) - \ell)}
         .
     \end{align*}
\end{proof}
    
Let $\Theta \colon X \times X \to \{0,1,2,\ldots \} \cup \{\infty\}$,
\[
    \Theta(x, x')
    = \inf \{k \geq 0 : T_{1,k}(x) = T_{1,k}(x') \}
    .
\]

\begin{lem}
    \label{lem:step-coupling}
    Let $n \geq 0$ and let $\mu_{A_{n-1}}$ and $\mu_{A_n}$ be
    the restrictions of $\mu$ on respective sets,
    normalized to probability.
    Then there exists a probability measure $\tmu$ on $X \times X$
    with marginals $\mu_{A_{n-1}}$ and $\mu_{A_n}$ such that
    for $\ell \geq 1$,
    \[
        \tmu( \Theta \geq \varkappa_n + \ell )
        \leq C
        \begin{cases}
            \ell^{-\beta + 1}, & n = 0,
            \\
            \min \{ \tau_{L_n}(A_n) \ell^{-\beta}, \ell^{-\beta + 1} \},
                               & n > 0 \text{ and } a_n \subset Y,
            \\
            0, & \text{else}.
        \end{cases}
    \]
\end{lem}

\begin{proof}
    First we assume that $n > 0$.
    Observe that if $a_n \subset X \setminus Y$, then $A_n$ = $A_{n-1}$
    and the result is clear. Suppose that $a_n \subset Y$. Note that then
    $\varkappa_n = n + \tau_{L_n}$.

    Since $\mu$ is regular, $(T_{1,n})_* \mu_{A_{n-1}}$ is supported on $Y$
    with $|(T_{1,n})_* \mu_{A_{n-1}}|_{\LL} \leq K_1$,
    and similarly $|(T_{1,\varkappa_n})_* \mu_{A_n}|_{\LL} \leq K_1$.

    Let $\mu' = (T_{1, \varkappa_n})_* \mu_{A_{n-1}}$ and
    $\mu'' = (T_{1, \varkappa_n})_* \mu_{A_n}$.
    By Remark~\ref{rmk:forward-tails}, both $\mu'$ and $\mu''$ are regular
    with tail bound $C \min\{ \tau_{L_n} \ell^{-\beta}, \ell^{-\beta + 1} \}$.
    By Theorem~\ref{thm:decdec}, there exist decompositions
    \begin{equation}
        \label{eq:co-decomp}
        \mu'  = \sum_{\ell \geq 1} \alpha_\ell \mu'_\ell
        \qquad \text{and} \qquad
        \mu'' = \sum_{\ell \geq 1} \alpha_\ell \mu''_\ell
    \end{equation}
    such that
    $(T_{\varkappa_n + 1, \varkappa_n + \ell})_* \mu'_\ell
    = (T_{\varkappa_n + 1, \varkappa_n + \ell})_* \mu''_\ell = m$
    and
    $\sum_{k \geq \ell} \alpha_k \leq C \min\{ \tau_{L_n} \ell^{-\beta}, \ell^{-\beta + 1} \}$.

    Write
    \[
        \mu_{A_{n-1}} = \sum_{\ell \geq 1} \alpha_\ell \mu_{A_{n-1},\ell}
        \qquad \text{and} \qquad
        \mu_{A_n}     = \sum_{\ell \geq 1} \alpha_\ell \mu_{A_n    ,\ell}
        ,
    \]
    where $(T_{1,\varkappa_n})_* \mu_{A_{n-1}, \ell} = \mu'_\ell$
    and   $(T_{1,\varkappa_n})_* \mu_{A_n,     \ell} = \mu''_\ell$.
   
    As in the proof of Corollary~\ref{cor:decdec}, for each $\ell$
    there is a probability measure $\tmu_\ell$ on $X \times X$
    with marginals $\mu_{A_{n-1}, \ell}$ and $\mu_{A_n, \ell}$,
    such that $\tmu_\ell(\Theta \leq \varkappa_n + \ell) = 1$.

    Let $\tmu = \sum_{\ell \geq 1} \alpha_\ell \tmu_\ell$. Then
    the marginals of $\tmu$ are $\mu_{A_{n-1}}$ and $\mu_{A_n}$,
    and 
    \[
        \tmu(\Theta \geq \varkappa_n + \ell)
        \leq \sum_{k \geq \ell} \alpha_k
        \leq C \min\{ \tau_{L_n} \ell^{-\beta}, \ell^{-\beta + 1} \}
        ,
    \]
    as required.

    It remains to treat the case $n = 0$. The proof is similar to above,
    only now $\mu_{A_{n-1}} = \mu$ and by Remark~\ref{rmk:forward-tails},
    both $\mu'$ and $\mu''$ are regular with tail bound $C \ell^{-\beta + 1}$.
    Thus we have the decomposition~\eqref{eq:co-decomp} with
    $\sum_{k \geq \ell} \alpha_k \leq C \ell^{-\beta + 1}$.
    The rest of the proof is unchanged.
\end{proof}

In order to bound $\tX_n$, we define random variables $I_n$ and $J_n$ by
\[
    I_n
    = \sum_{\ell \leq n} \htau_{\ell} \lambda^{-(n - \ell)}
\]
and
\begin{align*}
    J_n
    & = \sum_{j \geq 1} \Lip_{r_n + j - 1}(H)
    \min \bigl\{j^{-\beta + 1}, \tau_n j^{-\beta} \bigr\}
    \quad \text{for } n \geq 1,
    \\
    J_0
    & = \sum_{j \geq 1} \Lip_{r_0 + j - 1}(H) j^{-\beta + 1}
    .
\end{align*}

\begin{prop}
    \label{prop:tX}
    \[
        | \tX_n |
        \leq 
        \begin{cases}
            C ( I_{L_n} + J_{L_n} ), & n \in \{r_k\}_{k \geq -1},
            \\
            0, & \text{else}.
        \end{cases}
    \]
\end{prop}

\begin{proof}
    We bound $\tX_n(A_n)$ for $A_n \in \cB_n$ corresponding to the symbolic
    itinerary $a_0, \ldots, a_n$, as defined before.
    Let $\tmu$ be the measure on $X \times X$ from
    Lemma~\ref{lem:step-coupling}.
    For $(x,x') \in X \times X$, let $G(x,x') = (H(x), H(x'))$.
    Then
    \begin{equation}
        \label{eq:nol}
        \int G \, d\tmu
        = (\bE(H \mid A_{n-1}), \bE(H \mid A_n))
        .
    \end{equation}
    Let $x,x' \in A_{n-1}$; note that this holds for $\tmu$-almost every $(x, x')$.
    By Proposition~\ref{prop:back-cont},
    \\
    \begin{equation}
        \label{eq:raz}
        \sum_{k < \varkappa_n(A_n)} \Lip_k (H) d(T_{1,k}(x), T_{1,k}(x'))
        \leq C \sum_{\ell \leq L_n(A_n)} \htau_{\ell}(A_n) \lambda^{-(L_n(A_n) - \ell)}
        .
    \end{equation}
    From Lemma~\ref{lem:step-coupling},
    \begin{equation}
        \label{eq:razdvatri}
        \begin{aligned}
            & \sum_{k \geq \varkappa_n(A_n)}
            \int \Lip_k(H) d(T_{1,k}(x), T_{1,k}(x')) \, d\tmu(x,x')
            \leq C \sum_{k \geq \varkappa_n(A_n)} \Lip_k(H) \tmu (\Theta > k)
            \\
            & \leq C \sum_{j \geq 1} \Lip_{\varkappa_n(A_n) + j - 1}(H)
            \begin{cases}
                j^{-\beta + 1}, & n = 0,
                \\
                \min\{\tau_{L_n}(A_n) j^{-\beta}, j^{-\beta + 1} \}, & n > 0 \text{ and } a_n \subset Y,
                \\
                0, & \text{else}.
            \end{cases}
        \end{aligned}
    \end{equation}
    Recall that $\tX_n(A_n) = \bE(H \mid A_n) - \bE(H \mid A_{n-1})$.
    The combination of~\eqref{eq:nol},~\eqref{eq:raz} and~\eqref{eq:razdvatri}
    yields the desired bounds.
\end{proof}

\begin{rmk}
    \label{rmk:Itau}
    By Jensen's inequality,
    $I_n^2 \leq C \sum_{\ell \leq n} \htau_\ell^2 \lambda^{-(n-\ell)}$, so
    \[
        \sum_{n \geq 0} I_n^2
        \leq C \sum_{n \geq 0} \htau_\ell^2
        .
    \]
\end{rmk}

By Burkholder inequality (Theorem~\ref{thm:burk},~\ref{thm:burk:Lp} and~\ref{thm:burk:wLp}),
$\| H \|_p \leq C_p \bigl\| \sum_{n \geq 0} |\tX_n|^2 \bigr\|_p$
and
$\| H \|_{p, \infty} \leq C_p \bigl\| \sum_{n \geq 0} |\tX_n|^2 \bigr\|_{p, \infty}$
for each $p > 1$, with $C_p$ depending only on $p$.
Hence by Proposition~\ref{prop:tX} and Remark~\ref{rmk:Itau},
\begin{equation}
    \label{eq:Hp}
    \begin{aligned}
        \| H \|_p
        & \leq C C_p \Bigl\| \Bigl( \sum_{n \geq 0} \htau_n^2 \Bigr)^{1/2} \Bigr\|_p
        + C C_p \Bigl\| \Bigl( \sum_{n \geq 0} J_n \Bigr)^{1/2} \Bigr\|_p
        ,
        \\
        \| H \|_{p, \infty}
        & \leq C C_p \Bigl\| \Bigl( \sum_{n \geq 0} \htau_n^2 \Bigr)^{1/2} \Bigr\|_{p, \infty}
        + C C_p \Bigl\| \Bigl( \sum_{n \geq 0} J_n^2 \Bigr)^{1/2} \Bigr\|_{p, \infty}
        .
    \end{aligned}
\end{equation}
Since $\mu$ has tail bound $C \ell^{-\beta + 1}$ and returns to $Y$
are full branch Gibbs-Markov maps,
for all $n, \ell \geq 1$,
\begin{equation}
    \label{eq:tatu}
    \begin{aligned}
        \mu(\tau_0 \geq \ell)
        & \leq C \ell^{-\beta + 1}
        ,
        \\
        \mu(\tau_n \geq \ell \mid \tau_0, \ldots, \tau_{n-1})
        & \leq C \ell^{-\beta}
        .
    \end{aligned}
\end{equation}
We show separately, in Section~\ref{sec:mart}, that~\eqref{eq:tatu}
can be used to bound the right hand side of~\eqref{eq:Hp}
well enough to complete the proof of Theorem~\ref{thm:mom}.


\section{Quadratic variation}
\label{sec:mart}

In this section we bound quadratic variation for processes
driven by nonstationary renewal-like sequences with polynomial
renewal times, as those appearing in Section~\ref{sec:mom:new}.
The main results, Theorems~\ref{thm:mart} and~\ref{thm:omega},
and their proofs are an adaptation of the corresponding parts of~\cite{GM14},
with an improvement when $\beta = 2$.

This section is self-contained, and notation is unrelated to the rest
of the paper.

Let $a_n$, $n \geq 0$, be a nonnegative sequence.
(In notation of Section~\ref{sec:mom:new}, $a_n$ plays the role of $\Lip_n(H)$.)
Let $\tau_n$, $n \geq 0$ be a sequence of random variables with values in $\{1,2, \ldots\}$
such that with some $C_\tau > 0$ and $\beta > 1$:
\begin{align*}
    \bP(\tau_0 \geq \ell)
    & \leq C_\tau \ell^{-\beta + 1}
    \qquad \text{ for all } \ell \geq 1,
    \\
    \bP(\tau_n \geq \ell \mid \tau_0, \ldots, \tau_{n-1})
    & \leq C_\tau \ell^{-\beta}
    \qquad \text{ for all } n \geq 1 \text{ and } \ell \geq 1.
\end{align*}
Let $r_{-1} = 0$ and $r_n = \sum_{j \leq n} \tau_j$ for $n \geq 0$. Define
\[
    \sigma
    = \sum_{n \geq 0} (a_{r_{n-1}} + \cdots + a_{r_n - 1})^2
\]
and
\[
    \omega
    = \sum_{n \geq 0} \Bigl(
        \sum_{j \geq 1} a_{r_n + j - 1} \min\{ \tau_n j^{-\beta}, j^{-\beta + 1} \}
    \Bigr)^2
    .
\]

Recall the notation $\|\cdot\|_p$ and $\|\cdot\|_{p, \infty}$ as in~\eqref{eq:Lp-norms}.

\begin{thm}
    \label{thm:mart}
    There is a constant which depends only on $\beta$ and $C_\tau$
    such that:
    \begin{enumerate}[label=(\alph*)]
        \item If $\beta \in (1,2)$, then
            $ \displaystyle
                \| \sigma^{1/2} \|_{\beta, \infty}
                \leq C \Bigl( \sum_{n \geq 0} a_n^\beta \Bigr)^{1/\beta}
                .
            $
        \item If $\beta = 2$, then
            $ \displaystyle
                \| \sigma^{1/2} \|_2
                \leq C \Bigl( \sum_{n \geq 0} a_n^2 \bigl(1 + \log (n + 1)\bigr) \Bigr)^{1/2}
                .
            $
            In addition, for each $p > 2$,
            \[
                \| \sigma^{1/2} \|_p
                \leq C_p \Bigl( \sum_{n \geq 0} a_n^2 \bigl(1 + \log (n + 1)\bigr) \Bigr)^{1/2}
                + C_p \Bigl( \sum_{n \geq 0} a_n^2 \Bigr)^{1/p}
                \Bigl( \sum_{n \geq 0} a_n \Bigr)^{1 - 2/p}
                ,
            \]
            where $C_p$ depends only on $\beta$, $C_\tau$ and $p$.
        \item If $\beta > 2$, then
            $ \displaystyle
                \| \sigma^{1/2} \|_{2(\beta - 1)}
                \leq C \Bigl( \sum_{n \geq 0} a_n^2 \Bigr)^{1/2}
                .
            $
    \end{enumerate}
\end{thm}

\begin{thm}
    \label{thm:omega}
    There is a constant which depends only on $\beta$ and $C_\tau$
    such that:
    \begin{enumerate}[label=(\alph*)]
        \item If $\beta \in (1,2]$, then
            $ \displaystyle
                \| \omega^{1/2} \|_\beta
                \leq C \Bigl( \sum_{n \geq 1} a_n^\beta \Bigr)^{1/\beta}
                .
            $
        \item If $\beta > 2$, then
            $ \displaystyle
                \omega^{1/2}
                \leq C \Bigl( \sum_{n \geq 1} a_n^2 \Bigr)^{1/2}
                .
            $
    \end{enumerate}
\end{thm}

Proof of Theorems~\ref{thm:mart} and~\ref{thm:omega} takes the rest of this section.

\subsection{Proof of Theorem~\ref{thm:mart}}

A key ingredient of the proof is:

\begin{thm}[Burkholder type inequalities]
    \label{thm:burk}
    Suppose that $M_n$ is a martingale adapted to a filtration $\cF_n$
    with increments $X_n = M_n - M_{n-1}$,
    maximum $M^*_n = \max_{j \leq n} |M_n|$ and
    quadratic variation $[M]_n = \sum_{j \leq n} |X_j|^2$.
    Let $p > 1$ and let $c_p$ and $C_p$ denote constants
    which depend only on $p$. Then for all $n$:
    \begin{enumerate}[label=(\alph*)]
        \item\label{thm:burk:Lp}
            \(
                \displaystyle
                c_p \bigl\| [M]_n^{1/2} \bigr\|_p
                \leq \|M^*_n\|_p
                \leq C_p \bigl\| [M]_n^{1/2} \bigr\|_p
                .
            \)
        \item\label{thm:burk:wLp}
            \(
                \displaystyle
                c_p \bigl\| [M]_n^{1/2} \bigr\|_{p, \infty}
                \leq \|M^*_n\|_{p, \infty}
                \leq C_p \bigl\| [M]_n^{1/2} \bigr\|_{p, \infty}
            \).
        \item\label{thm:burk:p12}
            If $p \in (1,2)$, then
            \(
                \bigl\| [M]_n^{1/2} \bigr\|_{p, \infty}
                \leq C_p \bigl( \sum_{j \leq n} \|X_j\|_{p, \infty}^p \bigr)^{1/p}
                .
            \)
        \item\label{thm:burk:ros}
            \(
                \displaystyle
                \|M_n\|_p
                \leq C_p \Bigl\| \sum_{j \leq n} \bE \bigl( |X_j|^2 \mid \cF_{j-1} \bigr) \Bigr\|_{p/2}^{1/2}
                + C_p \bigl\| \max_{j \leq n} |X_j| \bigr\|_p
                .
            \)
    \end{enumerate}
\end{thm}

\begin{proof}
    Parts~\ref{thm:burk:Lp} and~\ref{thm:burk:ros} are proved
    in Burkholder \cite[Theorems~3.2 and~21.1]{B73}.
    Part~\ref{thm:burk:wLp} can be found in Johnson and Schechtman \cite[Remark~6]{JS88}.
    To prove part~\ref{thm:burk:p12}, write
    \[
        \bigl\|[M]_n^{1/2}\bigr\|_{p, \infty}
        = \bigl\| [M]_n^{p/2} \bigr\|_{1, \infty}^{1/p}
        \leq \Bigl( \frac{4}{2-p} \sum_{j \leq n} \bigl\| |X_j|^p \bigr\|_{1,\infty} \Bigr)^{1/p}
        = C_p \Bigl( \sum_{j \leq n} \bigl\|X_j\bigr\|_{p,\infty}^p \Bigr)^{1/p}
        ,
    \]
    where we used a surrogate triangle inequality for $\|\cdot\|_{1,\infty}$ from
    Vershynin~\cite[Proposition~1]{V12} (which is an extended version of Hagelstein~\cite[Theorem~2]{H05}).
    Alternatively, part~\ref{thm:burk:p12} is a corollary of
    part~\ref{thm:burk:wLp} and~\cite[Theorem~2.5]{GM14}.
\end{proof}

\begin{rmk}
    \label{rmk:sqmart}
    Let $a_{n,j} = \sum_{\ell=n}^{n+j-1} a_\ell$ and let
    $A_n = \sum_{0 \leq j \leq n} I_j a_{r_{j-1}, \tau_j}$, where $I_n$ are independent
    (also from $\tau_n$) coin flips, $\bP(I_n = \pm 1) = 1/2$.
    Then $A_n$ is a martingale with quadratic variation $[A]_\infty = \sigma$.
    With Theorem~\ref{thm:burk} this implies that for $p > 1$,
    \begin{equation}
        \label{eq:uch}
        \|\sigma^{1/2}\|_p
        \leq C_p \Bigl\|
        \sum_{n \geq 0} \bE \bigl( a_{r_{n-1}, \tau_n}^2 \mid \tau_{0}, \ldots, \tau_{n-1} \bigr)
        \Bigr\|_{p/2}^{1/2}
        + C_p \bigl\| \max_{n \geq 0} a_{r_{n-1}, \tau_n} \bigr\|_{p}
    \end{equation}
    and that for $p \in (1,2)$,
    \begin{equation}
        \label{eq:mik}
        \|\sigma^{1/2}\|_{p, \infty}^p
        \leq C_p \sum_{j \geq 0} \|a_{r_{n-1}, \tau_n}\|_{p, \infty}^p
        .
    \end{equation}
\end{rmk}

Another key ingredient for the case $\beta > 2$ is an elementary inequality with a surprisingly
nontrivial proof:

\begin{lem}[{\cite[Lemma~4.4]{GM14}}]
    \label{lem:fun}
    Suppose that $\beta > 2$. Consider a nonnegative sequence $w_n$ with
    $\sum_{k \geq n} w_k = O(n^{-\beta})$. There exists a constant $C$ such that
    for every sequence $a_n \in \ell^2(\bZ)$,
    \[
        \sum_{n \in \bZ} \sum_{k \geq 0} w_k \Bigl( \sum_{j=n-k}^{n+k} a_j \Bigr)^{2(\beta - 1)}
        \leq C \Bigl( \sum_{n \in \bZ} a_n^2 \Bigr)^{\beta - 1}
        .
    \]
\end{lem}

For $\beta = 2$ we use a simpler inequality:

\begin{lem}
    \label{lem:fun2}
    There is a constant $C > 0$ such that for every nonnegative sequence $a_n$,
    \[
        \sum_{n \geq 1} \sum_{k \geq 1} k^{-3} (a_n + \ldots + a_{n + k - 1})^2
        \leq C \sum_{n \geq 1} a_n^2 (1 + \log n)
        .
    \]
\end{lem}

\begin{proof}
    Write
    \begin{align*}
        \sum_{n \geq 1}
        & \sum_{k \geq 1} k^{-3} (a_n + \ldots + a_{n + k - 1})^2
        \leq \sum_{n \geq 1} \sum_{k \geq 1} k^{-2} (a_n^2 + \ldots + a_{n + k - 1}^2)
        \\
        & = \sum_{m \geq 1} a_m^2 \sum_{n \leq m} \sum_{k > m-n} k^{-2}
        \leq C \sum_{m \geq 1} a_m^2 (1 + \log m)
        .
    \end{align*}
\end{proof}

We use $C$ to denote various constants which depend only on $\beta$ and $C_\tau$.
As in Remark~\ref{rmk:sqmart}, we let $a_{n,j} = \sum_{\ell=n}^{n+j-1} a_\ell$.
Throughout we use the observation that if $b_j$ is an increasing sequence, then
\begin{align*}
    \sum_{j \geq 1} \bP(\tau_0 = j) b_j
    & \leq C \sum_{j \geq 1} j^{-\beta} b_j
    \quad \text{and}
    \\
    \sum_{j \geq 1} \bP(\tau_n = j \mid \tau_0, \ldots, \tau_{n-1}) b_j
    & \leq C \sum_{j \geq 1} j^{-\beta - 1} b_j
    \quad \text{for } n \geq 1
    .
\end{align*}

\subsubsection{Proof of Theorem~\ref{thm:mart}, case \texorpdfstring{$\beta \in (1,2)$}{beta in (1,2}.}

Let $M_n = \sup_{k \geq 0} \frac{a_{n-k} + \cdots + a_{n+k}}{2k + 1}$,
where $a_{n} = 0$ if $n  < 0$.
For $n \geq 1$,
\begin{align*}
    \|a_{r_{n-1}, \tau_n}\|_{\beta, \infty}^\beta
    & \leq \| (2 \tau_n + 1) M_{r_{n-1}} \|_{\beta, \infty}^\beta
    \\
    & = \sup_{t > 0} \bE \bigl( t^\beta \bP( M_{r_{n-1}} (2 \tau_n + 1) > t \mid \tau_0, \ldots, \tau_{n-1}) \bigr)
    \leq C \bE M_{r_{n-1}}^\beta
    .
\end{align*}
Next,
\[
    \sum_{n \geq 1} \|a_{r_{n-1}, \tau_n}\|_{\beta, \infty}^\beta
    \leq C \bE \sum_{n \geq 1} M_{r_{n-1}}^\beta
    \leq C \sum_{n \geq 0} M_{n}^\beta
    \leq C \sum_{n \geq 0} a_n^\beta
    ,
\]
where for the last step we used the Hardy-Littlewood maximal inequality (cf.~\cite[Theorem~2.3]{GM14}).
The term corresponding to $n=0$ is simpler:
\begin{align*}
    \|a_{0,\tau_0}\|_{\beta, \infty}^\beta
    & = \|a_{0,\tau_0}^\beta \|_{1, \infty}
    \leq \bigl\| \tau_0^{\beta - 1} ( a_0^\beta + \ldots + a_{\tau_0 - 1}^\beta ) \bigr\|_{1, \infty}
    \\
    & \leq \| \tau_0 \|_{\beta - 1, \infty}^{\beta - 1} ( a_0^\beta + a_{1}^\beta + \cdots )
    \\
    & \leq C ( a_0^\beta + a_{1}^\beta + \cdots )
    .
\end{align*}
Altogether,
\begin{equation}
    \label{eq:abetainf}
    \sum_{n \geq 0} \|a_{r_{n-1}, \tau_n}\|_{\beta, \infty}^\beta
    \leq C \sum_{n \geq 0} a_n^\beta
    .
\end{equation}
The desired result follows from~\eqref{eq:mik}.

\subsubsection{Proof of Theorem~\ref{thm:mart}, case \texorpdfstring{$\beta > 2$}{beta > 2}.}

We bound the two terms on the right hand side of~\eqref{eq:uch},
giving special treatment to the case $n = 0$.

First,
\begin{equation}
    \label{eq:param-pam-pam}
    \begin{aligned}
        \bE a_{0, \tau_0}^{2(\beta - 1)}
        & = \sum_{j \geq 1} \bP(\tau_0 = j) a_{0,j}^{2(\beta - 1)}
        \leq C \sum_{j \geq 1} j^{-\beta} a_{0,j}^{2(\beta - 1)}
        \\
        & \leq C \sum_{k \in \bZ} \sum_{j \geq 0} j^{-\beta - 1} (a_{k-j} + \cdots + a_{k+j})^{2(\beta - 1)}
        \leq C \Bigl( \sum_{j \geq 0} a_j^2 \Bigr)^{\beta - 1}
        ,
    \end{aligned}
\end{equation}
where $a_j = 0$ for $j < 0$, and for the last inequality we used Lemma~\ref{lem:fun}.
It follows that
\begin{equation}
    \label{eq:tu-tu-tu}
    \bE a_{0, \tau_0}^2
    \leq C \sum_{j \geq 0} a_j^2
    .
\end{equation}

Next, for $p \geq 1$,
\begin{equation}
    \label{eq:bca}
    \begin{aligned}
        \sum_{n \geq 1}
        & \bE \bigl( a_{r_{n-1}, \tau_n}^p \mid \tau_{0}, \ldots, \tau_{n-1} \bigr)
        = \sum_{n \geq 1} \sum_{j \geq 1}
        \bP \bigl( \tau_n = j \mid  \tau_{0}, \ldots, \tau_{n-1} \bigr) a_{r_{n-1},j}^p
        \\
        & \leq C \sum_{n \geq 1} \sum_{j \geq 1} j^{-\beta - 1} a_{r_{n-1},j}^p
        \leq C \sum_{n \geq 1} \sum_{j \geq 1} j^{-\beta - 1} a_{n,j}^p
        .
    \end{aligned}
\end{equation}
Since $a_{n,j}^2 \leq j (a_n^2 + \ldots + a_{n+j-1}^2)$, using~\eqref{eq:bca}
with $p=2$ yields
\[
    \sum_{n \geq 1} \bE \bigl( a_{r_{n-1}, \tau_n}^2 \mid \tau_{0}, \ldots, \tau_{n-1} \bigr)
    \leq C \sum_{n \geq 1} \sum_{k \geq 0} (k+1)^{-\beta} a_{n+k}^2
    \leq C \sum_{n \geq 1} a_n^2
    .
\]
Summing the above with~\eqref{eq:tu-tu-tu}, we obtain
\begin{equation}
    \label{eq:tu-tu}
    \sum_{n \geq 0} \bE \bigl( a_{r_{n-1}, \tau_n}^2 \mid \tau_{0}, \ldots, \tau_{n-1} \bigr)
    \leq C \sum_{n \geq 0} a_n^2
    .
\end{equation}

Write
\begin{align*}
    \bE \max_{n \geq 1} a_{r_{n-1}, \tau_n}^{2(\beta - 1)}
    & \leq \bE \sum_{n \geq 1} a_{r_{n-1}, \tau_n}^{2(\beta - 1)}
    \leq \bE \sum_{n \geq 1} \bE( a_{r_{n-1}, \tau_n}^{2(\beta - 1)} \mid \tau_0, \ldots, \tau_{n-1})
    \\
    & \leq C \sum_{n \geq 1} \sum_{j \geq 1} j^{-\beta-1} a_{n,j}^{2(\beta - 1)}
    \leq C \Bigl( \sum_{n \geq 1} a_n^2 \Bigr)^{\beta - 1}
    ,
\end{align*}
where for the second last inequality we used~\eqref{eq:bca}, and for the
last inequality we used~Lemma~\ref{lem:fun}. With~\eqref{eq:param-pam-pam},
we have
\begin{equation}
    \label{eq:label}
    \bE \max_{n \geq 0} a_{r_{n-1}, \tau_n}^{2(\beta - 1)}
    \leq C \Bigl( \sum_{n \geq 0} a_n^2 \Bigr)^{\beta - 1}
    .
\end{equation}

Altogether,~\eqref{eq:tu-tu},~\eqref{eq:label} and~\eqref{eq:uch} prove the desired bound.

\subsubsection{Proof of Theorem~\ref{thm:mart}, case \texorpdfstring{$\beta = 2$}{beta =2}.}
The proof is similar to that for $\beta > 2$, using Lemma~\ref{lem:fun2} instead of Lemma~\ref{lem:fun}.

Let $M_n = \max_{k < n} \frac{a_k + \cdots + a_{n-1}}{n - k}$. Then
\begin{equation}
    \label{eq:param-pam-pam2}
    \bE a_{0, \tau_0}^2
    = \sum_{j \geq 1} \bP(\tau_0 = j) a_{0,j}^2
    \leq C \sum_{j \geq 1} j^{-2} a_{0,j}^2
    \leq C \sum_{j \geq 1}  M_j^2
    \leq C \sum_{j \geq 0}  a_j^2
    ,
\end{equation}
where we used the Hardy-Littlewood maximal inequality at the last step
(taking into account that $M_n \leq 3 \max_{k > 0} \frac{a_{n-1-k} + \cdots + a_{n-1+k}}{2k + 1}$).

Next, similar to~\eqref{eq:bca} and using Lemma~\ref{lem:fun2},
\begin{equation}
    \label{eq:bcaz}
    \begin{aligned}
        \sum_{n \geq 1}
        & \bE \bigl( a_{r_{n-1}, \tau_n}^2 \mid \tau_{0}, \ldots, \tau_{n-1} \bigr)
        = \sum_{n,j \geq 1}
        \bP \bigl( \tau_n = j \mid  \tau_{0}, \ldots, \tau_{n-1} \bigr) a_{r_{n-1},j}^2
        \\
        & \leq C \sum_{n,j \geq 1} j^{-3} a_{r_{n-1},j}^2
        \leq C \sum_{n,j \geq 1} j^{-3} a_{n,j}^2
        \leq C \sum_{n \geq 1} a_n^2 (1 + \log n)
        .
    \end{aligned}
\end{equation}

By~\eqref{eq:param-pam-pam2} and~\eqref{eq:bcaz},
\begin{equation}
    \label{eq:bbd}
    \sum_{n \geq 0} \bE \bigl( a_{r_{n-1}, \tau_n}^2 \mid \tau_{0}, \ldots, \tau_{n-1} \bigr)
    \leq C \sum_{n \geq 0} a_n^2 (1 + \log (n + 1))
\end{equation}
and
\[
    \bE \bigl( \max_{n \geq 0} a_{r_{n-1}, \tau_n}^2 \bigr)
    \leq \bE
    \sum_{n \geq 0} \bE \bigl( a_{r_{n-1}, \tau_n}^2 \mid \tau_{0}, \ldots, \tau_{n-1} \bigr)
    \leq C \sum_{n \geq 0} a_n^2 (1 + \log (n + 1))
    .
\]
Hence by~\eqref{eq:uch},
\[
    \bE \sigma
    \leq C \sum_{n \geq 0} a_n^2 (1 + \log (n + 1))
    .
\]

It remains to bound $\bE \sigma^{p/2}$ with $p > 2$.
By~\eqref{eq:abetainf}, which is not restricted to $\beta \in (1,2)$,
\[
    \bigl\| \max_{n \geq 0} a_{r_{n-1}, \tau_n} \bigr\|_{2, \infty}^2
    \leq \sum_{n \geq 0} \| a_{r_{n-1}, \tau_n} \|_{2, \infty}^2
    \leq C \sum_{n \geq 0} a_n^2
    .
\]
Thus $\bP(\max_{n \geq 0} a_{r_{n-1}, \tau_n} \geq t) \leq C t^{-2} \sum_{n \geq 0} a_n^2$.
Also, $\|\max_{n \geq 0} a_{r_{n-1}, \tau_n} \|_\infty \leq \sum_{n \geq 0} a_n$, hence
\begin{align*}
    \bE \bigl( \max_{n \geq 0} a_{r_{n-1}, \tau_n} \bigr)^p
    & \leq C_p \sum_{n \geq 0} a_n^2
    \int_0^{\sum_{n \geq 0} a_n} t^{p-1} t^{-2} \, dt
    \\
    & \leq C_p \Bigl( \sum_{n \geq 0} a_n^2 \Bigr) \Bigl( \sum_{n \geq 0} a_n \Bigr)^{p-2}
    .
\end{align*}
Here and further, $C_p$ denotes constants which depend only on $\beta$, $C_\tau$ and $p$.
By the above,~\eqref{eq:bbd} and~\eqref{eq:uch},
\[
    \| \sigma^{1/2} \|_p
    \leq C_p \Bigl( \sum_{n \geq 0} a_n^2 (1 + \log (n + 1)) \Bigr)^{1/2}
    + C_p \Bigl( \sum_{n \geq 0} a_n^2 \Bigr)^{1/p}
    \Bigl( \sum_{n \geq 0} a_n \Bigr)^{1-2/p}
    ,
\]
as required.

\subsection{Proof of Theorem~\ref{thm:omega}}

We abbreviate $\delta_{\ell, k} = \min\{ \ell k^{-\beta}, k^{-\beta + 1} \}$.

The case $\beta > 2$ is simple: by Jensen's inequality,
\[
    \Bigl(
        \sum_{j \geq 1} a_{r_n + j - 1} \delta_{\tau_n,j}
    \Bigr)^2
    \leq C \sum_{j \geq 1} a_{r_n + j - 1}^2 j^{-\beta + 1}
    ,
\]
so
\[
    \omega
    \leq C \sum_{n \geq 0} \sum_{j \geq 1} a_{r_n + j - 1}^2 j^{-\beta + 1}
    \leq C \sum_{n \geq 1} a_n^2
    .
\]
Further we treat $\beta \in (1,2]$.

Let
\[
    \tomega
    = \sum_{n \geq 1} \Bigl(
        \sum_{j \geq 1} a_{r_n + j - 1} \delta_{\tau_n, j}
    \Bigr)^2
    ,
\]
so that
\[
    \omega
    = \Bigl(
        \sum_{j \geq 1} a_{r_0 + j - 1} \delta_{\tau_0, j}
    \Bigr)^2
    + \tomega
    .
\]
By H\"older's inequality,
\[
    \sum_{j \geq 1} a_{r_0 + j - 1} \delta_{\tau_n, j}
    \leq \Bigl( \sum_{j \geq 1} a_{r_0 + j - 1}^\beta \Bigr)^{1/\beta}
    \Bigl( \sum_{j \geq 1} j^{-\beta} \Bigr)^{(\beta - 1) / \beta}
    \leq C \Bigl( \sum_{j \geq 1} a_j^\beta \Bigr)^{1/\beta}
    ,
\]
hence it remains to show that
\begin{equation}
    \label{eq:tomega}
    \bE \tomega^{\beta/2}
    \leq C \sum_{j \geq 1} a_j^\beta
    .
\end{equation}

We note that
\[
    \tomega^{1/2}
    \leq \biggl[
        \sum_{n \geq 1} \Bigl(
        \sum_{j \geq 1} a_{r_n + j - 1} \delta_{\tau_n,j}
        \Bigr)^\beta
    \biggr]^{1/\beta}
\]
and thus
\begin{align*}
    \bE \tomega^{\beta/2}
    & \leq \bE \sum_{n \geq 1} 
    \bE \Bigl[ \Bigl( \sum_{j \geq 1} a_{r_n + j - 1} \delta_{\tau_n,j} \Bigr)^\beta
        \, \Bigm| \,
        \tau_0, \ldots, \tau_{n-1} 
    \Bigr]
    \\
    & = \bE \sum_{n, \ell \geq 1} \bP(\tau_n = \ell \mid \tau_0, \ldots, \tau_{n-1})
    \Bigl( \sum_{j \geq 1} a_{r_{n-1} + \ell + j - 1} \delta_{\ell,j} \Bigr)^\beta
    \\
    & \leq \bE \sum_{n, \ell \geq 1} p_{n, \ell}
    \Bigl( \sum_{j \geq 1} a_{n + \ell + j - 1} \delta_{\ell,j} \Bigr)^\beta
    ,
\end{align*}
where
\[
    p_{n, \ell} 
    = \begin{cases}
        \bP(\tau_k = \ell \mid \tau_0, \ldots, \tau_{k-1}), & r_{k-1} = n \text{ for some } k,
        \\
        0, & \text{else}.
    \end{cases}
\]
Recall that $\sum_{k \geq \ell} p_{n, k} \leq C \ell^{-\beta}$ almost surely
for all $n,\ell \geq 1$.
Now the bound~\eqref{eq:tomega} follows from Lemma~\ref{lem:4.5},
which is a version of~\cite[Lemma~4.5]{GM14}. This completes the proof
of Theorem~\ref{thm:omega}.

\begin{lem}
    \label{lem:4.5}
    Suppose that $p_{n,  \ell}$, $n, \ell \geq 1$ are nonnegative constants such that
    for all $\ell \geq 1$,
    \[
        \sup_n \sum_{k \geq \ell} p_{n, k}
        \leq C_\beta \ell^{-\beta},
    \]
    where $\beta > 1$ and $C_\beta > 0$. Then for every nonnegative sequence
    $a_n$,
    \[
        \sum_{n,\ell \geq 1} p_{n,\ell}
        \Bigl(
            \sum_{k \geq 1} a_{n + \ell + k} \min\{ \ell k^{-\beta}, k^{-\beta + 1} \}
        \Bigr)^\beta
        \leq C \sum_{n \geq 3} a_n^{\beta}
        .
    \]
\end{lem}

\begin{proof}
    In this proof $C$ denotes various constants which only depend on $\beta$ and $C_\beta$.
    We continue to abbreviate $\delta_{\ell, k} = \min\{ \ell k^{-\beta}, k^{-\beta + 1} \}$.

    We suppose that $\beta \in (1,2)$. The proof for $\beta \geq 2$ is similar and simpler.

    First we note a couple of simple bounds:
    \begin{equation}
        \label{eq:njj}
        \sum_{k \geq 1} \delta_{\ell, k}
        \leq C \ell^{2 - \beta}
    \end{equation}
    and for $\gamma > 2$,
    \begin{equation}
        \label{eq:njjj}
        \begin{aligned}
            \sum_{1 \leq \ell < m} \ell^{-\gamma} \delta_{\ell, m - \ell}
            & \leq C_{\beta,\gamma} \sum_{1 \leq \ell < m/2} \ell^{-\gamma} \delta_{\ell, m/2}
            + C_{\beta,\gamma} \sum_{m/2 \leq \ell < m} m^{-\gamma} \delta_{m, m - \ell}
            \\
            & \leq C_{\beta,\gamma} m^{-\beta} + C_{\beta,\gamma} m^{-\gamma + 1}
            ,
        \end{aligned}
    \end{equation}
    where $C_{\beta,\gamma}$ depends only on $\beta$ and $\gamma$.
    
    Let $\Phi \colon \bR^\bN \to \bR^{\bN \times \bN}$ be the linear operator
    \[
        \{a_n\}_{n \geq 1}
        \mapsto \Bigl\{
            \ell^{-(2-\beta)} \sum_{k \geq 1} a_{n + \ell + k} \delta_{\ell, k} 
        \Bigr\}_{n, \ell \geq 1}
        .
    \]
    We equip $\bN$ with the counting measure and $\bN \times \bN$ with the measure
    $\{(n, \ell)\} \mapsto p_{n, \ell} \ell^{\beta(2 - \beta)}$.
    In this formulation, it is enough to prove that $\Phi$ is bounded as an operator
    from $L^\beta$ to $L^\beta$ with the norm only depending on $C$ and $\beta$.
    To achieve this, we show that $\Phi$ is bounded from $L^1$ to $L^1$
    and from $L^\infty$ to $L^\infty$.
    Then the result follows from the Marcinkiewicz interpolation theorem.

    Boundedness from $L^\infty$ to $L^\infty$ is immediate due to~\eqref{eq:njj}, so it remains
    to prove boundedness from $L^1$ to $L^1$, i.e.\ to show that for nonnegative $a_n$,
    \begin{equation}
        \label{eq:ljj}
        \sum_{n, \ell, k \geq 1} p_{n, \ell} \ell^{(\beta - 1)(2-\beta)}
        a_{n + \ell + k} \delta_{\ell, k}
        \leq C \sum_{n \geq 3} a_n
        .
    \end{equation}
    Letting $n+\ell+k = j$ and $k+\ell = m$, we rewrite the left hand side above
    in terms of $j$, $m$ and $\ell$:
    \begin{equation}
        \label{eq:lljj}
        \cdots =
        \sum_{j \geq 3} a_j \sum_{2 \leq m < j} \sum_{1 \leq \ell < m} p_{j - m, \ell}
        \ell^{(\beta - 1)(2-\beta)}
        \delta_{\ell, m - \ell}
        .
    \end{equation}
    Since $\ell^{(\beta - 1)(2-\beta)} \delta_{\ell, m - \ell}$ is increasing with $\ell$
    and $\sum_{k \geq \ell} p_{n,k} \leq C_\beta \ell^{-\beta}$, 
    for an upper bound we replace $p_{j-m, \ell}$ with $C \ell^{-\beta - 1}$ for $\ell < m-1$,
    and $p_{j-m, m-1}$ with $C \ell^{-\beta}$:
    \begin{align*}
        \sum_{1 \leq \ell < m}
        & p_{j - m, \ell} \ell^{(\beta - 1)(2-\beta)}
        \delta_{\ell, m - \ell}
        \\
        & \leq C 
        \sum_{1 \leq \ell < m - 1} \ell^{-\beta - 1} \ell^{(\beta - 1)(2-\beta)}
        \delta_{\ell, m - \ell}
        + C (m-1)^{-\beta} (m-1)^{(\beta - 1)(2-\beta)} \delta_{m-1, 1}
        \\
        & \leq C m^{(\beta - 1)(\beta - 2) - \beta}
        .
    \end{align*}
    For the last inequality we used~\eqref{eq:njjj} with
    $\gamma = \beta + 1 - (\beta - 1)(2 - \beta) > 2$.
    Since $(\beta -1)(\beta - 2) - \beta < -1$,
    \[
        \sum_{2 \leq m < j} \sum_{1 \leq \ell < m} p_{j - m, \ell}
        \ell^{(\beta - 1)(2-\beta)}
        \delta_{\ell, m - \ell}
        \leq C \sum_{2 \leq m < j} m^{(\beta - 1)(\beta - 2) - \beta}
        \leq C
        .
    \]
    The $L^1$ to $L^1$ bound~\eqref{eq:ljj} follows from the above and~\eqref{eq:lljj}.

\end{proof}


\section*{Acknowledgements}

A.K.\ is supported by an Engineering and Physical Sciences Research Council grant
EP/P034489/1.
J.L. received funding from the European Research Council (ERC) under the
European Union’s Horizon 2020 research and innovation programme (grant agreement No 787304).
The authors thank Viviane Baladi, Mark Holland and universities of Exeter and Sorbonne
for support and hospitality during their visits.



\begin{thebibliography}{11}


\bibitem{AHNTV15} R.~Aimino, H.~Hu, M.~Nicol, A.~T\"or\"ok, S.~Vaienti,
\emph{Polynomial loss of memory for maps of the interval with a neutral fixed point,}
Discrete Contin. Dyn. Syst. \textbf{35} (2015), 793--806.

\bibitem{ANV15} R.~Aimino, M.~Nicol, S.~Vaienti,
\emph{Annealed and quenched limit theorems for random expanding dynamical systems,}
Probab. Theory Related Fields \textbf{162} (2015), 233--274.

\bibitem{AR16} R.~Aimino, J.~Rousseau,
\emph{Concentration inequalities for sequential dynamical systems of the unit interval,}
Ergodic Theory Dynam. Systems \textbf{36} (2016), 2384--2407.

\bibitem{BB16}  W.~Bahsoun, C.~Bose,
\emph{Mixing rates and limit theorems for random intermittent maps,}
Nonlinearity \textbf{29} (2016), 1417--1433.

\bibitem{BBD14} W.~Bahsoun, C.~Bose, Y.~Duan,
\emph{Decay of correlation for random intermittent maps,}
Nonlinearity \textbf{27} (2014), 1543--1554.

\bibitem{BBR19} W.~Bahsoun, C.~Bose, M.~Ruziboev,
\emph{Quenched decay of correlations for slowly mixing systems,}
Trans. Amer. Math. Soc. \textbf{372} (2019), 6547--6587.

\bibitem{BRS20} W.~Bahsoun, M.~Ruziboev, B.~Saussol,
\emph{Linear response for random dynamical systems,}
Adv. Math. \textbf{364} (2020), 107011.


\bibitem{B73} D.~L. Burkholder,
\emph{Distribution function inequalities for martingales,}
Ann. Probability \textbf{1} (1973), 19--42.

\bibitem{CR07} J.P.~Conze, A.~Raugi,
\emph{Limit theorems for sequential expanding dynamical systems on $[0,1]$,}
Ergodic theory and related fields, Contemp.  Math. \textbf{430} (2007), 89--121.

\bibitem{DGM18} J.~Dedecker, S.~Gou\"ezel, F.~Merlev\`ede,
\emph{Large and moderate deviations for bounded functions of slowly mixing Markov chains,}
Stoch. Dyn. \textbf{18} (2018), 1850017.

\bibitem{DM15} J.~Dedecker, F.~Merlev\`ede,
\emph{Moment bounds for dependent sequences in smooth Banach spaces,}
Stochastic Process. Appl. \textbf{125} (2015), 3401--3429.

\bibitem{DS16} N.~Dobbs, M.~Stenlund,
\emph{Quasistatic dynamical systems,}
Ergodic Theory Dynam. Systems \textbf{37} (2016), 2556--2596.

\bibitem{DFGTV18.1} D.~Dragi\v{c}evi\'c, G.~Froyland, C.~Gonz\'aez-Tokman, S.~Vaienti,
\emph{Almost sure invariance principle for random piecewise expanding maps,}
Nonlinearity \textbf{31} (2018), 2252–2280.

\bibitem{DFGTV18.2} D.~Dragi\v{c}evi\'c, G.~Froyland, C.~Gonz\'alez-Tokman, S.~Vaienti,
\emph{A spectral approach for quenched limit theorems for random expanding dynamical systems,}
Comm. Math. Phys. \textbf{360} (2018), 1121--1187.

\bibitem{DFGTV20} D.~Dragi\v{c}evi\'c, G.~Froyland, C.~Gonz\'alez-Tokman, S.~Vaienti,
\emph{A spectral approach for quenched limit theorems for random hyperbolic dynamical systems,}
Trans. Amer. Math. Soc. \textbf{373} (2020), 629--664.

\bibitem{FFV18} A.C.M.~Freitas, J.M.~Freitas, S.~Vaienti,
\emph{Extreme Value Laws for sequences of intermittent maps,}
Proc. Amer. Math. Soc. \textbf{146} (2018), 2103--2116.

\bibitem{G04c} S.~Gou\"ezel,
\emph{Central limit theorem and stable laws for intermittent maps,}
Probab. Theory Related Fields \textbf{128} (2004), 82--122.


\bibitem{G04} S.~Gou\"ezel, 
\emph{Sharp polynomial estimates for the decay of correlations,}
Israel J. Math. \textbf{139} (2004), 29--65.

\bibitem{GM14} S.~Gou\"ezel, I.~Melbourne,
\emph{Moment bounds and concentration inequalities for slowly mixing dynamical systems,}
Electron. J. Probab. \textbf{19} (2014), 93.

\bibitem{H19} Y.~Hafouta,
\emph{A vector valued almost sure invariance principle for time dependent non-uniformly expanding dynamical systems,}
arXiv:1910.12792 (2019).

\bibitem{H05} P.A.~Hagelstein,
\emph{Weak $L^1$ norms of random sums,}
Proc. Amer. Math. Soc. \textbf{133} (2005), 2327--2334.

\bibitem{HNTV17} N.~Haydn, M.~Nicol, A.~T\"or\"ok, S.~Vaienti,
\emph{Almost sure invariance principle for sequential and non-stationary dynamical systems,}
Trans. Amer. Math. Soc. \textbf{369} (2017), 5293--5316.

\bibitem{HRY17} N.~Haydn, J.~Rousseau, F.~Yang,
\emph{Exponential law for random maps on compact manifolds,}
arXiv:1705.05869 (2017).

\bibitem{JS88} W.B.~Johnson, G.~Schechtman,
\emph{Martingale inequalities in rearrangement invariant function spaces,}
Israel J. Math. \textbf{3} (1988), 267--275.


\bibitem{KKM18} A.~Korepanov, Z.~Kosloff, I.~Melbourne,
\emph{Martingale-coboundary decomposition for families of dynamical systems,}
Ann. Inst. H. Poincar\'e Anal. Non Lin\'eaire \textbf{35} (2018), 859--885.

\bibitem{KKM19} A.~Korepanov, Z.~Kosloff, I.~Melbourne,
\emph{Explicit coupling argument for nonuniformly hyperbolic transformations,}
Proc. Roy. Soc. Edinbourgh Sect. A. \textbf{149} (2019), 101--130.

\bibitem{L17} J.~Lepp\"anen,
\emph{Functional correlation decay and multivariate normal approximation for non-uniformly expanding maps,}
Nonlinearity \textbf{30} (2017), 4239--4259.

\bibitem{L18} J.~Lepp\"anen,
\emph{Intermittent quasistatic dynamical systems: weak convergence of fluctuations,}
Nonauton. Dyn. Syst. \textbf{5} (2018), 8--34.

\bibitem{LS16} J.~Lepp\"anen, M.~Stenlund,
\emph{Quasistatic dynamics with intermittency,}
Math. Phys. Anal. Geom. \emph{19} (2016), 8.

\bibitem{Li79} T.~Lindvall,
\emph{On Coupling of Discrete Renewal Processes,}
Z. Wahrscheinlichkeitstheor. verw. Geb. {\bf 48} (1979), 57--70.

\bibitem{LSV99} C.~Liverani, B.~Saussol, and S.~Vaienti,
\emph{A probabilistic approach to intermittency,}
Ergodic Theory Dynam. Systems \textbf{19} (1999), 671--685.

\bibitem{M09} I.~Melbourne,
\emph{Large and moderate deviations for slowly mixing dynamical systems,}
Proc. Amer. Math. Soc. \textbf{137} (2009), 1735--1741.

\bibitem{NPT19} M.~Nicol, F.P.~Pereira, A.~T\"or\"ok,
\emph{Large deviations and central limit theorems for sequential and random systems of intermittent maps,}
Ergodic Theory Dynam. Systems, to appear; arXiv:1909.07435 (2019).

\bibitem{NTV18} M.~Nicol, A.~T\"or\"ok, S.~Vaienti,
\emph{Central limit theorems for sequential and random intermittent dynamical systems,}
Ergodic Theory Dynam. Systems \textbf{38} (2018), 1127-1153.

\bibitem{PS09} M.~Pollicott, R.~Sharp,
\emph{Large deviations for intermittent maps,}
Nonlinearity \textbf{22} (2009), 2079--2092.

\bibitem{PM80} Y.~Pomeau, P.~Manneville,
\emph{Intermittent transition to turbulence in dissipative dynamical systems,}
Comm. Math. Phys. \textbf{74} (1980), 189--197.

\bibitem{S84} R.M.~Shortt,
\emph{Universally  measurable  spaces:  an  invariance  theorem  and  diversecharacterizations,}
Fund. Math. \textbf{121} (1984), 169--176.

\bibitem{SVZ20} M.~Stadlbauer, P.~Varandas, X.~Zhang,
\emph{Quenched and annealed equilibrium states for random Ruelle expanding maps and applications,}
arXiv:2004.04763 (2020).

\bibitem{SW71} E.M.~Stein, G.~Weiss,
\emph{Introduction to Fourier analysis on Euclidean spaces,}
Princeton Univ. Press (1971).

\bibitem{SYZ13} M.~Stenlund, L.~Young, H.~Zhang,
\emph{Dispersing billiards with moving scatterers,}
Comm. Math. Phys. \textbf{322} (2013), 909--955.

\bibitem{S19} Y.~Su,
\emph{Vector-valued almost sure invariance principle for non-stationary dynamical systems,}
arXiv:1903.09763 (2019).

\bibitem{V12} R.~Vershynin,
\emph{Weak triangle inequalities for weak $L^1$ norm,}
\url{https://www.math.uci.edu/~rvershyn/papers/weak-L1.pdf}.

\end{thebibliography}
\end{document}